\documentclass[11pt,onecolumn,letterpaper]{report}
\usepackage{amsmath,amsthm,amssymb,amsfonts,verbatim,latexsym}
\usepackage[dvips]{graphicx}
\usepackage[latin1]{inputenc}
\usepackage{enumerate}
\usepackage{enumitem}
\usepackage{multicol}
\usepackage{graphicx}
\usepackage{makeidx}
\usepackage{pstricks}
\usepackage{psfrag}
\usepackage{hyperref}

\usepackage{accents} 	

\makeindex
 \addtocounter{page}{-0}

\begin{document}

\title{Selected Topics in Random Walks in Random Environment}
\author{
 Alexander Drewitz 
\thanks{Department of Mathematics, Columbia University,  Broadway 2990, New York City, New York 10027, USA;}
\and Alejandro F. Ram\'\i rez
\thanks{Facultad de Matem\'{a}ticas, Pontificia Universidad Cat\'{o}lica de Chile, Vicu\~{n}a Mackenna 4860, Macul, Santiago, Chile;}$^{,\,}$\thanks{
Partially supported by Fondo Nacional de Desarrollo Cient\'\i fico
y Tecnol\'ogico grant 1100298.}
}
\date{September 8, 2013}
\maketitle

%
%
%
%
%
%
%
 \begin{abstract}
Random walk in random environment (RWRE) is a fundamental model of statistical mechanics, describing the movement 
of a
particle in a highly disordered and inhomogeneous medium as a
 random walk with random jump probabilities. It has been introduced 
in a series of papers as a model of DNA chain replication and crystal growth
(see Chernov \cite{Ch-67} and Temkin \cite{Te-69,Te-72}), and also as a model of turbulent behavior in
fluids through a Lorentz gas description (Sina{\u\i} 1982 \cite{Si-82b}). It is a simple but powerful model
 for a variety of complex large-scale disordered phenomena arising from
 fields such as physics, biology
and engineering. While the one-dimensional model is well-understood, in the multidimensional setting, fundamental
 questions about the RWRE model have resisted repeated and persistent
 attempts to answer them. Two major complications in this context stem from the loss of the Markov property under the averaged
 measure as well as the fact that in dimensions larger than one, the RWRE is not reversible anymore.
In these notes we present a general
overview of the model, with an emphasis on the
multidimensional setting and a more detailed
description of recent progress around ballisticity
questions.
 \end{abstract}

  \newcommand{\ii}{\'{\i}}

\newcommand{\set}[1]{\left\{#1\right\}}

\newcommand{\parcial}[2]{\frac{\partial}{\partial #2} \, #1}
\newcommand{\diff}[2]{\frac{d}{d #2} \, #1}

\newcommand{\Sob}{{\mathcal W}}
\newcommand{\Qset}{{\mathbb Q}}
\newcommand{\Kset}{{\mathbb K}}
\newcommand{\Zset}{{\mathbb Z}}
\newcommand{\Nset}{{\mathbb N}}
\newcommand{\Rset}{{\mathbb R}}
\newcommand{\Cset}{{\mathbb C}}
\newcommand{\Tset}{{\mathcal T}}

\newcommand{\intt}{\!\!\!{}^{{}^{{}^{\circ}}}}
\newcommand{\ppp}{\longrightarrow\!\!\!\!\!\!\!\!\!\!\!_{_{k \to \infty}}^{^{(b)}}}
\newcommand{\pps}{-\!\!\! -\!\!\! \longrightarrow\!\!\!\!\!\!\!\!\!\!\!\!\!\!\!^{^{\mathcal{S}(\Rset^{n})}}}
\newcommand{\ppmu}{\longrightarrow\!\!\!\!\!\!\!\!^{^{\mu \quad}}}
\newcommand{\ppn}{\longrightarrow\!\!\!\!\!\!\!\!\!\!\!_{_{n \to \infty \,\,}}}
\newcommand{\ppk}{\longrightarrow\!\!\!\!\!\!\!\!\!\!\!_{_{k \to \infty}}}
\newcommand{\HI}{\!\!\!\!\!\!^{^{HI}}}
\newcommand{\Dnz}{\frac{d^{n}}{dz^{n}}}
\newcommand{\esp}{\textup{E}}
\newcommand{\prob}{\textup{P}}
\newcommand{\ind}{\upharpoonleft\!\!\!|}
\newcommand{\ord}{\textup{ord}}
\newcommand{\nequiv}{\equiv \!\!\!\!\! / \,\, }

\newcommand{\nin}{\notin}
\newcommand{\Ree}{\textup{Re}\,}
\newcommand{\Imm}{\textup{Im}\,}
\newcommand{\supp}{\textup{supp}\,}
\newcommand{\dist}{\textup{dist}\,}

\renewcommand{\chaptername}{Chapter}
\renewcommand{\proofname}{Proof}
\newcommand{\Pbox}{\ensuremath{(\mathcal P)}}
\newcommand{\Pasymp}{\ensuremath{(\mathcal P^*)}}

\newfont{\fuentea}{cmsy10 scaled 2000}


\theoremstyle{definition}
\newtheorem{definition}{Definition}[chapter]
\theoremstyle{plain}
\newtheorem{theorem}[definition]{Theorem}
\newtheorem{lemma}[definition]{Lemma}
\newtheorem{corollary}[definition]{Corollary}
\newtheorem{proposition}[definition]{Proposition}
\newtheorem{conjecture}[definition]{Conjecture}
\newtheorem{remark}[definition]{Remark}
\newtheorem{example}[definition]{Example}
\newtheorem{question}[definition]{Open question}

\theoremstyle{definition}
\newtheorem*{problem}{Problem}
\newtheorem*{exercise}{Exercise}
\newtheorem*{observation}{Observation}
\newtheorem*{solution}{Solution}

\renewcommand{\theenumi}{(\roman{enumi})}
\renewcommand{\labelenumi}{\theenumi}


\makeatletter
\renewenvironment{proof}[1][\proofname]{\par
  \pushQED{\qed}%
  \normalfont \topsep6\p@\@plus6\p@\relax
  \trivlist
  \item[\hskip\labelsep
        \textsc
    #1\@addpunct{.}]\ignorespaces
}{%
  \popQED\endtrivlist\@endpefalse
}
\makeatother


\makeatletter
\renewenvironment{solution}[1][Solution]{\par
  \pushQED{\qed}%
  \normalfont \topsep6\p@\@plus6\p@\relax
  \trivlist
  \item[\hskip\labelsep
        \sc
    #1\@addpunct{.}]\ignorespaces
}{%
  \popQED\endtrivlist\@endpefalse
} \makeatother


\makeatletter
\def\@roman#1{\romannumeral #1}
\makeatother


  \tableofcontents

\chapter{Preface}

We present a review of random walks in random environment. The main focus evolves
around several fundamental open questions concerning the existence of
invariant probability measures, transience, recurrence, directional transience
and ballisticity. This choice of topics is somewhat
biased towards our recent research interests.

The first chapter deals with the question of the existence of
an invariant probability measure of the so-called ``environmental process''; such a measure is particularly useful if it is absolutely continuous
with respect to the  law of the environment. The existence and properties 
of such a measure characterize in some sense the different
asymptotic behaviors of the walk, from a general law of large
numbers to possibly a quenched central limit theorem, and to a
variational formula for the rate function in the case of quenched
large deviations. After the introduction of basic definitions and concepts,
we review the one-dimensional situation, which
turns out to be a controlled laboratory of several phenomena
which one would expect to encounter in the multidimensional setting.
Subsequently, we investigate the latter setting and give some of the corresponding (limited) results which are available in that context.

It is conjectured that for uniformly elliptic and i.i.d. environments,
in dimensions $d\ge 2$, directional transience implies ballisticity.
The second chapter of these notes reviews this question as well as
the progress and understanding which have been achieved towards its resolution.
In particular, we introduce the fundamental concept of renewal times. We then
proceed to the  ballisticity conditions, under which
it has been possible obtain a better understanding of the so-called slowdown phenomena 
as well as of the ballistic and diffusive behavior in the setting of (uniformly) elliptic environments.

\chapter{The environmental process and its invariant measures}
\label{chap:I}

\section{Definitions}

Throughout these notes, 
for $x\in\mathbb R^d$, we will use the notations
$|x|_\infty$, $|x|_1$ and $|x|_2$ for the $L^\infty$, $L^1$ and
$L^2$ norms.
For a subset $A \subset \mathbb Z^d$ we denote by
$
\partial A
$
its external boundary 
\begin{equation} \label{eqs:extBd}
 \big \{ x \in \mathbb Z^d \backslash A \, : \, \exists y \in A \text{ with } \vert x-y \vert_1 = 1 \big \},
\end{equation}
and for a subset $B \subset \mathbb R^d$ we denote by $\accentset{\circ} B$ its interior.
We write
\begin{equation} \label{eqs:ball}
B_p(x_0,r)=\{x\in\mathbb R^d: |x-x_0|_p\le r\}
\end{equation}
for the closed ball centered in the $x_0$ with radius $r$ in the $p$-norm.
In addition, set $B_p(r) := B_p(0,r).$
Furthermore, the set
$
U:=\{e\in\mathbb Z^d:|e|_1=1\}
$
will serve as the set of possible jumps for the random walk to be defined.
We will use $C$ to denote constants that can change from one side of an inequality to another, and $c_1, c_2, \ldots$
for constants taking fixed values. Furthermore, if we want to emphasize the dependence of a constant on quantities, such as e.g. the dimension,
we write $C(d).$
We begin with the definition of an environment.

\begin{definition} {\bf(Environment)} We define
the set 
\begin{equation} \label{eqs:transKernels}
\mathcal P := \Big \{ (p(e))_{e \in U} \in [0,1]^U \,:\, \sum_{e \in U} p(e)=1 \Big\}
\end{equation}
of $2d$-vectors $p$
serving as admissible transition probabilities.
An {\it environment} is an element $\omega$ of
the {\it environment space} $\Omega:=\mathcal P^{\mathbb Z^d}$
so that $\omega:=(\omega(x))_{x\in\mathbb Z^d}$, where
$\omega(x)\in\mathcal P$. We denote the components of
$\omega(x)$ by $\omega(x,e)$.
\end{definition}

\noindent Let us now define  a random walk
in a given environment $\omega$.

\begin{definition} {\bf (Random walk in an environment $\omega$)}
Let $\omega \in \Omega$ be an environment and let 
$\mathcal G$ be the $\sigma$-algebra on $(\mathbb Z^d)^{\mathbb N}$ defined by the cylinder functions.
For $x \in \mathbb Z^d,$ we define the
{\it random walk in the environment $\omega$ starting in $x$} as the Markov
chain $(X_n)_{n\ge 0}$ on $\mathbb Z^d$ whose
law $P_{x,\omega}$ on $((\mathbb Z^d)^{\mathbb N}, \mathcal G)$ is characterized by
\begin{align*}
P_{x,\omega} [X_0 = x] &= 1, \quad \text{and}\\
P_{x,\omega}[X_{n+1}=y+e\, | \, X_n=y] &= 
\left\{
 \begin{array}{ll}
\omega(y, e), \quad &\text{if } e \in U,\\
0, \quad &\text{otherwise},
\end{array}
\right.
\end{align*} 
whenever  $P_{x,\omega}[X_n = y] > 0,$ and $0$ otherwise.
Furthermore, we denote by
\begin{equation} \label{eqs:nstep}
p^{(n)}(x,y,\omega) := P_{x,\omega} [X_n = y]
\end{equation}
the $n$-step transition probability of the random walk
in the environment $\omega$.
\end{definition}

\noindent We will now account for the randomness in the environment.
For that purpose, let us endow the environment space $\Omega$
with the product topology and let $\mathbb P$ be some
probability measure defined on $(\Omega,\mathcal B(\Omega));$ here, $\mathcal B$ denotes the corresponding
Borel $\sigma$-algebra.
We call $\mathbb P$ the {\it law of the environment}
and for every measurable function $f$ defined on $\Omega,$ we denote by $\mathbb E[f]$ the corresponding expectation
if it exists.
It will frequently be useful to assume that
\begin{enumerate} [label={\bf (IID)}, ref={\bf (IID)}]
 \item \label{items:IIDassumption}
 the coordinate maps on the product space $\Omega$ are
independent and identically distributed (i.i.d.) under $\mathbb P.$
\end{enumerate}
In order to give a relaxation of \ref{items:IIDassumption} we introduce the following notation.
For each $y\in\mathbb Z^d$, let us denote by
$t_y$ the translation defined on the environment space $\Omega$
by
$$
(t_y\omega)(x,e):=\omega(x+y,e),
$$
for every $x\in\mathbb Z^d$ and $e\in U$.
It will often be useful to assume that
\begin{enumerate} [label={\bf (ERG)}, ref={\bf (ERG)}]
 \item \label{items:ERGassumption}
the family of transformations $(t_x)_{x\in\mathbb Z^d}$
is an ergodic family acting on
$(\Omega,{\mathcal B}(\Omega ),\mathbb P)$. 
\end{enumerate}
In other words,
if $A\in{\mathcal B}(\Omega)$ is such that
$A=t_x^{-1}(A)$ for every $x\in\mathbb Z^d$, then
$\mathbb P(A)=0$ or $1$. This condition is also called
{\it total ergodicity}.
In particular, note that 
\ref{items:IIDassumption} implies \ref{items:ERGassumption}.

For a fixed realization of $\omega,$ we now call $P_{x,\omega}$ the {\it quenched law} of the random
walk in random environment (RWRE). Using Dynkin's theorem, it is not hard to show that
for each $x \in \mathbb Z^d$ and $G\in\mathcal G,$
the mapping
$$
\omega \mapsto P_{x,\omega}[G]
$$
is $\mathcal B(\Omega)$-measurable. We can therefore define
on the space $(\Omega \times {(\mathbb Z^d)}^{\mathbb N},
\mathcal B(\Omega) \otimes \mathcal G)$ for each $x\in\mathbb Z^d$ the semi-direct
product $P_{x,\mathbb P}$ of the measures $\mathbb P$ and $P_{x,\omega}$ by
the formula

\begin{equation} \label{eqs:averagedMeasureDef}
P_{x,\mathbb P}[F\times G]:=\int_F P_{x,\omega}(G)\, \mathbb P({\rm d}\omega).
\end{equation}
We denote by $P_x$ the marginal law of $P_{x,\mathbb P}$
 on ${(\mathbb Z^d)}^{\mathbb N}$ and call it the {\it averaged} or {\it annealed} law
of the RWRE.
One of the difficulties arising in the study of RWRE is that under the
averaged law it is generally not Markovian
anymore.

We will need the concepts of ellipticity and uniform ellipticity.

\begin{definition} {\bf (Ellipticity and uniform ellipticity)}
Let $\mathbb P$ be a probability measure defined on the space
of environments $(\Omega, \mathcal B(\Omega)).$ 

\begin{itemize}

\item We say that $\mathbb P$ is {\it
elliptic} if 
\begin{enumerate} [label={\bf (E)}, ref={\bf (E)}]
\item \label{items:Eassumption}
for every $x\in\mathbb Z^d$
we have that

\begin{equation} \label{eqs:ellipt}
\mathbb P\big[ \min_{e \in U} \omega(x,e)>0 \big]=1.
\end{equation}
\end{enumerate}

\item We say that $\mathbb P$ is {\it uniformly
elliptic} if 

\begin{enumerate} [label={\bf (UE)}, ref={\bf (UE)}]

 \item \label{items:UEassumption}
there exists a constant $\kappa>0$
such that for every $x\in\mathbb Z^d$
we have that

\begin{equation} \label{eqs:unifEllipt}
\mathbb P \big[ \min_{e \in U} \omega(x,e)\ge\kappa \big]=1.
\end{equation}

\end{enumerate}
\end{itemize}

We will usually call the environment (uniformly) elliptic in that case also.

\end{definition}

\begin{remark}
This labeling is motivated by operator theory where one has analogous definitions
of elliptic and uniformly elliptic differential operators.
\end{remark}

The following auxiliary process will play a significant role in what follows.

\begin{definition} \label{defs:envProc} {(\bf Environment viewed from the particle)}.
Let $(X_n)$ be a RWRE. We define the {\it environment viewed from the particle}
(or also {\it the environmental process})
as the discrete time process

$$
\bar\omega_n:=t_{X_n}\omega,
$$
for $n\ge 0$, with state space $\Omega$.
\end{definition}

Apart from taking values in a compact state space, another advantage of the
environment viewed from the particle is that even under the averaged measure
it is Markovian, as is shown in the next result following Sznitman \cite{BS-02};
however,
the cost is that we now deal with an infinite dimensional state space.

\begin{proposition} Consider a RWRE in an environment
 with law $\mathbb P$.
Then, under $P_0$, the process $(\bar \omega_n)$
is Markovian with state space $\Omega,$ initial law $\mathbb P,$ and transition
kernel

\begin{equation}
\label{transition-kernel}
Rf(\omega):=\sum_{e\in U}\omega(0,e) f(t_e\omega),
\end{equation}
defined for $f$ bounded measurable on $\Omega$ and initial law $\mathbb P$.
\end{proposition}
\begin{proof} Let us first note that for every $x\in\mathbb Z^d$,
and every bounded measurable function $f$ on $\Omega$,

\begin{equation}
\label{e1}
E_{x,\omega}[f(\bar \omega_1)]=
E_{x,\omega}[f( t_{X_1}\omega)]=\sum_{e\in U} \omega(x,e)f(t_{x+e}\omega)
=\sum_{e\in U} t_x\omega(0,e)f(t_{e}(t_x\omega))
=Rf( t_x\omega ).
\end{equation}
Let now $f_i, i=0,\ldots,n+1$ be bounded measurable
functions. Note that
\begin{align*}
E_{0,\omega}[f_{n+1}(\bar\omega_{n+1}(
f_{n}(\bar\omega_{n})\cdots f_{0}(\bar\omega_{0})]
&=
E_{0,\omega}[f_{n+1}(t_{X_{n+1}}\omega)
\cdots f_{0}(t_{{X_0}}\omega)]\\
&
=E_{0,\omega}[E_{X_n,\omega}(f_{n+1}(t_{X_{1}}\omega))f_n(\bar\omega_n)
\cdots f_{0}(\bar\omega_0)]\\
&=E_{0,\omega}[Rf_{n+1}(\bar\omega_n)f_n(\bar\omega_n)
\cdots f_{0}(\bar\omega_0)],
\end{align*}
where in the second equality we took advantage of the Markov property of $(X_n)$ under $P_{0,\omega},$ and
in the last step we have used (\ref{e1}).
Since $Rf_{n+1}(\bar\omega_n)$ is $\mathcal F_n$-measurable,
where $\mathcal F_n$ is the natural filtration of $(\bar\omega_n)$,
it follows from the above that

\begin{equation}
\label{e2}
E_{0,\omega}[f_{n+1}(\bar\omega_{n+1}) \, | \, \bar\omega_0,\ldots,\bar\omega_n]
=Rf_{n+1}(\bar\omega_n),
\end{equation}
which proves the Markov property of the chain $(\bar\omega_n)$ under
the measure $P_{0,\omega}$. 
It follows that
  the  transition kernel for the quenched process 
is given by (\ref{transition-kernel}).
 Integrating 
$P_{0,\omega}$ with respect to $\mathbb P$ we finish the proof.
\end{proof}

\section{Invariant probability measures of the environment as seen from the random walk}

We now want to examine the invariant  measures of the
Markov chain $(\bar\omega_n)$. Given an arbitrary
probability measure $\mathbb P$ on $\Omega$, 
we define the probability measure $\mathbb P R$
through the identity


$$
\int R f d\mathbb P=\int fd(\mathbb P R),
$$
for every bounded continuous function $f$ on $\Omega.$ 
 Whenever
$
\mathbb P=\mathbb P R,
$
we will say that $\mathbb P$ is an {\it invariant probability measure} for the
environmental process. We will also need to
consider the possibility of having invariant measures which
are not necessarily probability measures: similarly to the above, we will
say that a measure $\nu$ is invariant for the environmental process if
for every bounded continuous function $f$ one has that
$$
\int fd\nu=\int Rfd\mathbb \nu.
$$
It is obvious that any degenerate probability measure which is translation
invariant, is an invariant probability measure: this corresponds to
any simple random walk. 
The following lemma is a standard result, but shows that there
might be some other ways of constructing more interesting invariant
probability measures. Recall that given
 a sequence of probability measures  a limit measure is defined
as the limit of any convergent
subsequence. 

\begin{lemma} Consider a RWRE and the corresponding
environmental process $(\bar\omega_n)$. Then,
if $\mathbb P$ is any probability measure
in $\Omega$, there exists at least one limit measure of the C\'esaro means

\begin{equation}
\label{cesaro-means}
\frac{1}{n+1}\sum_{i=0}^n \mathbb P R^i.
\end{equation}
Furthermore, every limit measure of this C\'esaro means
is an invariant probability measure for the Markov chain $(\bar\omega_n)$.
\end{lemma}
\begin{proof} Let $\mathbb P$ be an arbitrary probability measure
defined on the space $\Omega$. Denote for each $n\ge 0$
as $\nu_n$ the C\'esaro means of (\ref{cesaro-means}).

Since the space of probability measures defined on $\Omega$
is compact under the topology of weak convergence, we can
extract a weakly convergent subsequence $\nu_{n_k}$,
so that the C\'esaro means has at least one limit point
$\nu$.
 We claim that $\nu$ is
an invariant probability measure. Indeed, it is enough to prove that

$$
\int Rfd\nu=\int fd\nu
$$
for every bounded continuous function $f$.  But
since the transition kernel $R$ maps bounded and continuous functions to bounded and continuous functions, we have that
\begin{align*}
\int Rfd\nu&=\lim_{k\to\infty}\int Rfd\nu_{n_k} 
=\lim_{k\to\infty}\frac{1}{n_k+1}\sum_{i=0}^{n_k}\int fd(\nu R^{i+1})\\
&=
\lim_{k\to\infty}\left(\frac{1}{n_k+1}\sum_{i=0}^{n_k}\int fd(\nu R^i)
+   \frac{1}{n_k+1}\int fd(\nu R^{n_k+1})-
 \frac{1}{n_k+1}\int fd\nu\right)\\
&= \lim_{k\to\infty}\int fd\nu_{n_k}
=\int f d\nu.
\end{align*}

\end{proof}

\medskip

\noindent Knowing only the existence of an invariant probability
measure turns out not to be very helpful. We will see that 
what we really need is to find one which is absolutely continuous with respect
to the law $\mathbb P$ of the environment.

\begin{example}  Let us consider the case $d=1$. Assume \ref{items:Eassumption} to be fulfilled and define
\begin{equation} \label{eqs:rhoDef}
\rho(x,\omega):=\frac{\omega(x,-1)}{\omega(x,1)} \quad \text{and} \quad \rho(\omega) := \rho(0, \omega).
\end{equation}
If $\mathbb E[\rho]<1$ and the environment $(\omega(x))_{x\in\mathbb Z}$
is i.i.d. under the law $\mathbb P$,
 we will prove in this lecture that

$$
\nu (d\omega):=f(\omega)\mathbb P ( d\omega),
$$
where 

$$
f(\omega):=C
\left(1+\rho(0,\omega)\right)
\big(1+\rho(1,\omega)+\rho(1,\omega)\rho(2,\omega)
+\rho(1,\omega)\rho(2,\omega)\rho(3,\omega)+\cdots
\big) < \infty,
$$
for some constant $C>0$,
is an invariant probability measure for the process $(\bar \omega_n)$, cf. also Theorem \ref{thm:Kozlov} below.
\end{example}

\section{Transience and recurrence in the one-dimensional model}

The focus of this section will be on one-dimensional RWRE under the assumption \ref{items:Eassumption} and ergodicity properties of the law $\mathbb P$ of the environment.
In this context, we will derive explicit necessary and sufficient conditions in terms of the environment
for the walk being transient or recurrent.
It turns out that in this case the model is
reversible in the following sense: for $\mathbb P$-a.a. environments 
$\omega$ it is possible to find a
 measure defined on $\mathbb Z$ for which the random walk
$(X_n)$ in environment $\omega$ is reversible. This observation partly explains the fact that
many explicit computations can be performed, and even
explicit conditions characterizing particular behaviors of the walk
can be found.

The following lemma of Kesten \cite{Ke-75} will prove useful.

\medskip
\begin{lemma}
\label{kesten-lemma}
Given any stationary sequence of random variables
$(Y_n)_{n\ge 0}$ with law $P$ such that $P[\lim_{n\to\infty}\sum_{k=0}^nY_n=\infty]=1$ one has
$P\left[\liminf_{n\to\infty}\frac{1}{n}\sum_{k=0}^nY_k>0\right]=1$. 
\end{lemma}

\noindent In what follows, we will say that a function $f$ is
{\it Lebesgue integrable in the extended sense} if
its Lebesgue integral exists, possibly taking the values
$\infty$ or $-\infty$.
\medskip

\begin{theorem} 
 \label{thm:oneDimCharacterisation}
Consider a RWRE in dimension $d=1$ in  an  environment
with law $\mathbb P$ such that \ref{items:Eassumption} holds. Assume \ref{items:ERGassumption}
and that $\mathbb E[\log \rho]$ is Lebesgue integrable
in the extended sense.
 Then the following are satisfied.

\begin{enumerate}

\item \label{items:transR} If $\mathbb E[\log\rho]<0$ then the random walk is $P_0$-a.s.
transient to the right, i.e.,

$$
\lim_{n\to\infty}X_n=\infty,\qquad P_0-a.s.
$$

\item \label{items:transL} If $\mathbb E[\log\rho]>0$ then the random walk is $P_0$-a.s.
transient to the left, i.e.,

$$
\lim_{n\to\infty}X_n=-\infty,\qquad P_0-a.s.
$$

\item \label{items:rec} If $\mathbb E[\log\rho]=0$ then the random walk is $P_0$-a.s.
recurrent and

$$
\limsup_{n\to\infty}X_n=\infty\qquad{\rm and}\qquad
\liminf_{n\to\infty}X_n=-\infty\qquad P_0-a.s.
$$

\end{enumerate}

\end{theorem}

\medskip

\noindent The above theorem
was first proved within the context of branching processes in i.i.d. random
environments by Smith and Wilkinson in 1969 \cite{SW-69} (see also \cite[Remark 8]{KZ-13} and the references therein). In 1975 it was
proved by Solomon \cite{So-75} for
i.i.d. environments and afterwards extended 
to ergodic environments by Alili \cite{Al-99}.  Here, we present a proof
based on the method of Lyapunov functions (see
Comets, Menshikov and Popov \cite{CMP-98} and Fayolle, Malyshev
and Menshikov \cite{FMM-95}).
The so-called Sina{\u\i}'s regime corresponds
to the recurrent case under the additional assumption
that $0<\mathbb E[(\log\rho)^2]<\infty$. In \cite{Sin-82}, Sina{\u\i}
proved that under these conditions the position of the walk at
time $n$ is typically  of order $(\log n)^2$ under $P_0$.
We will see in section \ref{section-inv-meas}, that the dichotomy expressed
by Theorem \ref{thm:oneDimCharacterisation} is an expression of the different possibilities
concerning the existence of an invariant measure (not necessarily
a probability measure) for the environmental
process which is absolutely continuous with respect to $\mathbb P$: \ref{items:transL}
and \ref{items:transR} occur when there exists such a measure; \ref{items:rec} occurs
when such a measure does not exist.

\begin{proof} 
We want to find a martingale defined
in terms of the environment which discriminates
between transience and recurrence through the use
of the martingale convergence theorem. Let us furthermore try to find such a martingale 
of the form $f(X_n)$, where

$$
f(x)=\sum_{j=0}^{x-1}\Delta_j,
$$
for $x\ge 0$, and for some sequence $(\Delta_j)$ which will be chosen
appropriately.
In fact, using this convergence we will deduce the desired asymptotics from the properties of the limit
of that martingale. Now note that with $q(x) := \omega(x, 1)$ and $p(x) := \omega(x,-1),$

$$
E_{x,\omega}[f(X_{n+1})-f(X_n)\, | \, X_n = y]=
\left\{\begin{array}{ccc}
-p(y)\Delta_{y-1}+q(y)\Delta_y,
\qquad &\textrm{if}&\quad y\ge 2,\\
p(1)\Delta_{-1}+q(1)\Delta_1,
\qquad &\textrm{if}&\quad y=1,\\
p(y)\Delta_{y-2}-q(y)\Delta_{y-1},
\qquad &\textrm{if}&\quad y\le 0.
\end{array} \right.
$$
But if $f(X_n)$ is a martingale, the left-hand side of this display must 
vanish and we should have that

$$
\Delta_1=-\rho_1\Delta_{-1},
$$
and that

\begin{eqnarray*}
&\Delta_y=\rho_y\Delta_{y-1}\qquad {\rm for}\quad y\ge 2,\\
&\Delta_{y-2}=\rho^{-1}_y\Delta_{y-1}\qquad {\rm for}\quad y\le 0, \\
\end{eqnarray*}
where we have used the shorthand notation $\rho_y:=\rho(y,\omega)$.
Choosing $\Delta_0=-1$, $\Delta_1=-\rho_1$ and $\Delta_{-1}=1$  we deduce that

$$
f(x)=\left\{\begin{array}{ccc}
-\sum_{0\le j\le x-1}\prod_{i=1}^{j}\rho_i,
\qquad &\textrm{if}&\quad x\ge 0,\\
\sum_{x\le j\le -1}\prod_{i=j+1}^{0}\rho^{-1}_i,
\qquad &\textrm{if}&\quad x<0,
\end{array} \right.
$$
serves our purposes, where $\prod_{i=1}^{0}\rho_i:=1$.
Hence,  $f$ is harmonic with respect to the generator of the quenched RWRE
and $f(X_n),$
is an $\mathcal G_n$-martingale under the probability
measure $P_{0,\omega}$, where $\mathcal G_n$ is the natural $\sigma$-algebra
generated by the random walk. Now, by the ergodic theorem,
we have $\mathbb P$-a.s. that
$$
\prod_{i=1}^x\rho_i=\exp\left\{x(\mathbb E[\log\rho]+o(1))\right\},
$$
 as $x\to\infty,$
while when $x\to-\infty$, one has

$$
\prod_{i=x+1}^{-1}\rho_i=\exp\left\{x(\mathbb E[\log\rho]+o(1))\right\}.
$$
We now see that in the case $\mathbb E[\log\rho]<0$, 
there is a constant $C>0$ such that $\mathbb P$-a.s.

\begin{equation}
\label{limc}
\lim_{x\to\infty}f(x)=-C,
\end{equation}
and

$$
\lim_{x\to-\infty}f(x)=\infty.
$$
It follows that $\mathbb P$-a.s.

$$
E_{0,\omega}[f(X_n)_-]=\sum_{x=1}^\infty f(x)P_{0,\omega}[X_n=x]<\infty.
$$
By the martingale convergence theorem 
$P_0$-a.s.

\begin{equation}
\label{mart}
\lim_{n\to\infty}f(X_n) \quad \text{exists.}
\end{equation}
 Now, by ellipticity, it is easy to see that $P_0$-a.s., only the following three possibilities
can occur:

\begin{enumerate}

\item \label{items:infSupInfty} $\limsup_{n\to\infty}X_n=\infty$  and  $\liminf_{n\to\infty}X_n=-\infty$.

\item \label{items:limInfty} $\lim_{n\to\infty}X_n=\infty$.

\item \label{items:limMinusInfty} $\lim_{n\to\infty}X_n=-\infty$.

\end{enumerate}
By (\ref{mart}) and (\ref{limc}) we conclude that
necessarily case \ref{items:limInfty} above occurs. By a similar analysis we
see that if $\mathbb E[\log\rho]>0$, case \ref{items:limMinusInfty} happens.
Let us now consider the case
$$
\mathbb E[\log\rho]=0.
$$
If $\rho$ was almost surely constant and hence equal to $1,$ the above setting would be reduced to simple random walk, for which the 
corresponding result is canonical knowledge.
Therefore, without loss of generality, we can assume that
$\mathbb E[(\log\rho)^2]>0,$ and equally that $\mathbb P[\log\rho>0]>0$. Then,
by Lemma \ref{kesten-lemma} and the ergodicity of $\mathbb P$, we can
conclude that $\mathbb P$-a.s.,
$$
\limsup_{x\to\infty}\sum_{i=1}^x\log\rho_i>-\infty.
$$
 It follows that $\mathbb P$-a.s. one has that

$$
\lim_{x\to\infty}f(x)=-\infty,
$$
and similarly that

$$
\lim_{x\to-\infty}f(x)=\infty.
$$
If we define for $A>0$ the stopping times $T_A:=\inf\{k\ge 0: X_k\ge A\}$
and $S_A:=\inf\{k\ge 0: X_k\le -A\}$, we  see that
$f(X_{n\land T_A})$ and $f(X_{n\land S_A})$ are martingales
such that $E_{0,\omega} [f(X_{n\land T_A})_+ ] <\infty$
and $E_{0,\omega}[f(X_{n\land S_A})_-]<\infty$, respectively.
Hence, by the martingale convergence
theorem we conclude that the limits
$$
\lim_{n\to\infty}f(X_{n\land T_A}),\qquad\lim_{n\to\infty}f(X_{n\land S_A}),
$$
exist. The only possibility is that $\mathbb P$-a.s. we have that $P_{0,\omega}$-a.s.,
$X_n$ eventually hits both $A$ and $-A$. Since $A$ was chosen arbitrarily,
this proves part \ref{items:limMinusInfty} of the theorem.

\end{proof}

\section{Computation of an
absolutely continuous invariant measure in dimension $d=1$}
\label{section-inv-meas}
In 1999, Alili \cite{Al-99} proved a one-dimensional result which
establishes the existence of an invariant measure for the 
environment process as seen from the random walk with respect to
the initial law of the environment. 
The proof we present here, is due to Conze and Guivarc'h 
\cite{CG-00} (see also \cite{Rk-11}).
We will say that {\bf (B+)} is satisfied if

$$
\mathbb E\Big[(1+\rho_0)\sum_{j=0}^\infty\prod_{k=0}^{j-1}\rho_{k+1}\Big]<\infty,
$$
while we will say that {\bf (B-)} is satisfied if

$$
\mathbb E\Big[(1+\rho^{-1}_0)\sum_{j=0}^\infty
\prod_{k=-1}^{-j}\rho^{-1}_{k}\Big]<\infty.
$$
Note that in the i.i.d. case {\bf (B+)} reduces to $\mathbb E[\rho_0]<1$
while {\bf (B-)} to $\mathbb E[\rho^{-1}_0]<1$.

\medskip

\begin{theorem}
\label{theorem-alili} (Alili) Consider a RWRE 
with law $\mathbb P$ fulfilling \ref{items:Eassumption} and \ref{items:ERGassumption} in dimension $d=1$.
Then the following holds.

\begin{enumerate}

\item \label{items:noInvM} Assume that $\mathbb E[\log\rho]=0$. If
$\mathbb E[(\log\rho)^2]>0$, then there are no invariant measures
which are absolutely continuous with respect to $\mathbb P$.
If $\mathbb E[(\log\rho)^2]=0$, $\mathbb P$ is the unique invariant
measure of the environmental process absolutely continuous with
respect to $\mathbb P$ (up to multiplicative constants).

\item \label{items:absContM}
If $\mathbb E[\log\rho]>0$
but {\bf (B+)} is not satisfied, or
if $\mathbb E[\log\rho]<0$ but {\bf (B-)} is not satisfied,
the environment
viewed from the random walk has a unique invariant measure $\nu$
(up to multiplicative constants)
which is absolutely continuous with respect to $\mathbb P$,
but which is not a probability measure.

\item \label{items:absContProbM} If {\bf (B+)} is satisfied,
there exists a unique invariant probability measure $\nu$
which is absolutely continuous with respect to $\mathbb P$. Furthermore,

$$
\frac{d\nu}{d\mathbb P}=C(1+\rho_0)\sum_{j=0}^\infty\prod_{k=0}^{j-1}\rho_{k+1},
$$
for some constant $C>0$.

\item \label{items:absContProbMII} If {\bf (B-)} is satisfied,
there exists a unique invariant probability measure $\nu$
which is absolutely continuous with respect to $\mathbb P$.
Furthermore

$$
\frac{d\nu}{d\mathbb P}=C(1+\rho^{-1}_0)
\sum_{j=0}^\infty\prod_{k=-1}^{-j}\rho_{k}^{-1},
$$
for some constant $C>0$.
\end{enumerate}
\end{theorem}

\noindent We will see soon how this result 
exhibits a relationship  between
the existence of an absolutely continuous invariant probability measure and the ballisticity
of the random walk: in dimension $d=1$, the existence of an 
absolute continuous invariant
probability measure is equivalent to ballisticity. We will give
more details about this soon.

\medskip
 
\noindent {\it Sketch of the proof.}
We will start  proving part \ref{items:absContM}.
Note that if $\nu$ is an invariant measure, we have that
for every bounded measurable function $f$

$$
\int \left(q(0,\omega)f(t_1\omega)+p(0,\omega)f(t_{-1}\omega)\right)
\, \nu(d\omega) =\int f(\omega) \, \nu(d\omega).
$$
Now  if $\nu$ is absolutely continuous with respect to
$\mathbb P$ with density $\phi$, the above
equation is equivalent to

$$
q(0,t_{-1}\omega)\phi(t_{-1}\omega)+p(0,t_1\omega)\phi(t_1\omega)=\phi(\omega)
$$
holding for $\nu$-a.a. $\omega.$
We then have that

$$
h\circ t_1^2-\left(\frac{1}{1-q} h\right)\circ t_1+\rho^{-1} h=0,
$$
where $h:=p\phi$ and where we have written $p=p(0,\omega)$ and
$q=q(0,\omega)$. If we now define

$$
\widetilde h:=h\circ t_1-\rho^{-1} h,
$$
we conclude that for every $x\in\mathbb Z$,

$$
\widetilde h \circ t_x-\widetilde h=0.
$$
But since $\mathbb P$ is ergodic with respect to $(t_x)_{x\in\mathbb Z}$, we conclude that $\widetilde h$ is $\mathbb P$-a.s. 
equal to a constant $C$. Assume that $C=0$.
Then $\widetilde h = 0$ is equivalent to

$$
h(t_{1}\omega)=\rho^{-1}(\omega)h(\omega).
$$
We claim that the only solution in this case is $h=0$.
Indeed, using induction on $n$ we have that

$$
h(t_{n}\omega)=h(\omega)\prod_{j=0}^{n-1}\rho^{-1}(t_{j}\omega).
$$
If $\mathbb E[\log\rho]>0$, by the ergodic theorem this would imply that
a.s.

$$
\lim_{n\to\infty}h(t_n\omega)=0.
$$
Now integrating with respect to $\mathbb P$,
 using its stationarity and the fact that $h(t_n\omega),$ $n \in \mathbb N,$ are
uniformly integrable, we conclude that

\begin{equation}
\nonumber
\int h(\omega)\mathbb P(d\omega)=\lim_{n\to\infty}
\int h(t_{n}\omega)\mathbb P(d\omega)=0,
\end{equation}
so that

$$
h=0.
$$
Using a similar argument one arrives at the same conclusion
when $\mathbb E[\log\rho]<0$.

So let us assume that $C\ne 0$.  In this case we have that

\begin{equation}
\label{rec1}
h=(\rho^{-1}h)\circ t_{-1}+C.
\end{equation}
Now choose a constant $h_0$ and define recursively

\begin{equation} \label{eqs:Brecurs}
h_{n+1}:=(\rho^{-1}h_n)\circ t^{-1}_1+C.
\end{equation}
If we can prove that $h_n$ converges $\mathbb P$-a.s. as $n\to\infty$,
then the limit should be a solution to (\ref{rec1}). Now
from \eqref{eqs:Brecurs}
we can deduce
$$
h_n(\omega)=C\sum_{j=0}^{n-1}\prod_{k=0}^{j-1} \rho^{-1}(t^{-1}_{k+1}\omega) 
+\left(\prod_{k=1}^n \rho^{-1}(t^{-1}_k\omega) \right)h_0.
$$
Taking the limit when $n\to\infty$, we conclude 
 in the case in which $\mathbb E [\log\rho] > 0$, in combination with the ergodic theorem, that
$h_n$ converges $\mathbb P$-a.s. to
$$
h(\omega)=c\sum_{j=0}^\infty\prod_{k=0}^{j-1}\rho^{-1}(t^{-1}_{k+1}\omega).
$$
Thus,

$$
\phi(\omega)=(1+\rho^{-1}(\omega))\sum_{j=0}^\infty
\prod_{k=-1}^{-j}\rho^{-1}(t_{k}\omega).
$$
This proves part \ref{items:absContM} of the proposition. To prove part \ref{items:absContProbM},
note that Jensen's inequality and {\bf (B-)} imply that
$\mathbb E[\log\rho]<\infty$. Therefore, the measure with density $\phi$
already defined can be normalized to define a probability measure.
Similarly, one can prove part \ref{items:absContProbMII}.  The proof of part \ref{items:noInvM}
in the case $\mathbb E[(\log\rho)^2]>0$
is analogous to the proof of the recurrent case of Theorem  
\ref{thm:oneDimCharacterisation}.
The case $\mathbb E[(\log\rho)^2]=0$ is trivial, since in this case we would be in the situation of simple random walk.

\medskip

\section{Absolutely continuous invariant measures and some implications}

\noindent The existence of an invariant probability measure which is
absolutely continuous with respect to the initial distribution of the
environment will turn out to be crucial in the study of the model.
We recall that the environmental process has been defined in Definition \ref{defs:envProc}, which
considered as
a trajectory has state space $\Gamma:=\Omega^{\mathbb N}.$ Furthermore, define the law
 $P_\omega$ defined on its Borel $\sigma$-algebra $\mathcal B(\Gamma)$
 through the identity
\begin{equation} \label{eqs:environmentProcProb}
P_\omega [A]:=P_{0,\omega}[(\bar\omega_n)\in A],
\end{equation}
for any Borel subset $A$ of $\Gamma$ endowed with the product topology.
Furthermore, for any probability measure $\nu$ defined in $\Omega$,
we define

\begin{equation}
\label{law-environment}
P_\nu:=\int P_\omega \nu(d\omega).
\end{equation}
We will denote by $\theta:\Gamma\to\Gamma$ the canonical
shift on $\Gamma$ defined by

\begin{equation} \label{eqs:shiftDef}
\theta(\omega_0,\omega_1,\ldots):=(\omega_1,\omega_2,\ldots).
\end{equation}
The following result of Theorem \ref{thm:Kozlov} was proved by Kozlov in \cite{Ko-85}.
For its proof, we will follow Sznitman in \cite{BS-02}.
\begin{theorem} (Kozlov) \label{thm:Kozlov} Consider a RWRE in an environment with law
  $\mathbb P$ fulfilling \ref{items:Eassumption} and \ref{items:ERGassumption}.
Assume that there exists an invariant probability measure $\nu$ for the environment
seen from the random walk which is absolutely 
continuous with respect to $\mathbb P$. Then the following
are satisfied:

\begin{enumerate}

\item \label{items:nuEquiv} $\nu$ is equivalent to $\mathbb P$.

\item \label{items:nuErgod} The environment as seen from the random walk with initial
law $\nu$ is ergodic.

\item  \label{items:nuUniqInvCont} $\nu$ is the unique invariant probability measure 
for the environment as seen from the particle which is
absolutely continuous with respect to $\mathbb P$.

\item \label{last-part}The C\'esaro means

$$
\frac{1}{n+1}\sum_{i=0}^n \mathbb P R^i
$$
converges weakly to $\nu$.
\end{enumerate}
\end{theorem}

\noindent {\it Proof of part \ref{items:nuEquiv}}. Let $f$ be the Radon-Nikodym
derivative of $\nu$ with respect to $\mathbb P$ and consider the event
$E:=\{f=0\}$. 
In order to prove the desired result it will be sufficient to show
$
\mathbb P[E] = 0.
$

Since $\nu$ is invariant, we have that
$$
\int f \cdot (R1_E) \, d\mathbb P=(\nu R)[E]=\nu [E]=\int_{\{f=0\}} \, d\mathbb P=0.
$$
It follows that $\mathbb P$-a.s. on the event $E^c=\{f>0\}$ one has that
$R1_E=0$. Therefore, using the fact that $R1_E\le 1$, one has that
for every $e\in U$,

$$
1_E(\omega)\ge R1_E(\omega)=\sum_{e'\in U}\omega(0,e')1_{E}(t_{e'}\omega)
\ge \omega(0,e) 1_{E}(t_{e}\omega),\qquad\mathbb P-a.a. \, \omega.
$$
From the ellipticity assumption and the fact that $1_E(\omega)$ and $1_E(t_e\omega)$  for $e\in U$ only take the values $0$ or $1$
we have that for such $e$,

$$
1_E(\omega)\ge 1_E(t_e\omega),\qquad\mathbb P-a.s.
$$
Now using the fact that $\mathbb P[E]=\mathbb P[t_e^{-1}E]$ we conclude that
for each $e\in U$ one has
$$
1_E=1_{t^{-1}_eE},\qquad\mathbb P-a.a. \, \omega.
$$
%
Thus, we iteratively obtain that for each $x\in\mathbb Z^d,$
$$
1_E=1_{t^{-1}_x(E)},\qquad\mathbb P-a.s.
$$
It follows that the event

$$
\widetilde E:=\bigcap_{x\in\mathbb Z^d}t^{-1}_x(E),
$$
is invariant under the action of the family $(t_y)_{y\in\mathbb Z^d}$ and that it differs
form the event $E$ on an event of $\mathbb P$-probability $0$.
Since $\mathbb P$ is ergodic with respect to the family $(t_y)_{y\in\mathbb Z^d}$ we conclude
that

\begin{equation} \label{eqs:E01}
\mathbb P[E]=\mathbb P[\widetilde E] \in \{0,1\}.
\end{equation}
But since $\int_{E^c}f\, d\mathbb P=\int f\, d\mathbb P=1$ we know that $\mathbb P[E^c] > 0,$ which in combination with
$\mathbb P[\Omega] =1$  and \eqref{eqs:E01} implies $\mathbb P[E]=0$. Hence,
$\mathbb P$ is equivalent to $\nu$.

\medskip

\noindent {\it Proof of part \ref{items:nuErgod}}. 
We will prove that if $A\in\mathcal B(\Gamma)$ is invariant
so that $\theta^{-1}(A)=A$ then 
$P_\nu[A]$ (cf. (\ref{law-environment}) and \eqref{eqs:shiftDef})
is equal to $0$ or $1$.
 For $\omega \in \Omega$ define
$$
\phi(\omega):=P_\omega[A].
$$
We claim that
$$
(\phi(\bar \omega_n))_{n\ge 0}
$$
is a $P_\nu$-martingale with the canonical filtration on
$\Gamma$. In fact, note that since $A$ is invariant,
we have that $1_A=1_A\circ\theta_n$ and hence,
\begin{equation}
\label{e3}
 E_\nu[1_A \, | \, \bar \omega_0,\ldots, \bar \omega_n]=
 E_\nu[1_A\circ\theta_n \, |\, \bar \omega_0,\ldots, \bar \omega_n]
= P_{\bar \omega_n}[A]=\phi(\bar \omega_n),\qquad P_\nu-a.a. \, (\bar \omega_n).
\end{equation}
It follows from (\ref{e3}) and the martingale convergence theorem
that

\begin{equation}
\label{e31}
\lim_{n\to\infty}\phi (\bar \omega_n) =1_A((\bar \omega_n)_{n \in \mathbb N}) ,\qquad P_\nu-a.a. \, (\bar \omega_n)
\end{equation}
Let us now prove that there is a set $B\in\mathcal B(\Omega)$ such
that $\nu$-a.s.

\begin{equation}
\label{e4}
\phi=1_B.
\end{equation}
In fact, assume that (\ref{e4}) is not satisfied. Then there is
an interval $[a,b]\subset (0,1)$ with $a<b$ such that

\begin{equation}
\label{e5}
\nu[\phi\in [a,b]]>0.
\end{equation} Also, by the ergodic theorem we have that
$P_\nu$-a.s.

$$
\lim_{n\to\infty}\frac{1}{n}\sum_{k=0}^{n-1}
1_{\phi^{-1} ([a,b])}(\bar \omega_k) =\Psi:= E_\nu[
1_{\phi^{-1} ([a,b])}(\bar \omega_0) \, | \, \mathcal I],
$$
where $\mathcal I:=\{A\in\Gamma:\theta^{-1}(A)=A\}$ is the $\sigma$-field of invariant events.
Now, by (\ref{e5}),

$$
E_\nu[\Psi]= P_\nu[\phi (\omega_0)\in [a,b]]
=\nu[\phi\in [a,b]]>0.
$$
But this contradicts (\ref{e31}).
Hence, (\ref{e4}) holds. Let us now
prove that $\nu$-a.s.

\begin{equation}
\label{e6}
R1_B=1_B.
\end{equation}
Indeed, we have  that $P_\nu$-a.s. it is true that
$$
1_B(\omega_0)=E_\nu[1_B(\omega_1) \, |\, \omega_0]
=R1_B(\omega_0).
$$
Since $ P_\nu[A]=\nu[B]$, it is then enough to prove that

\begin{equation}
\label{e7}
\nu[B]\in \{0,1\}.
\end{equation}
Now, $\mathbb P$-a.s. we have that

$$
1_B(\omega)=R1_B(\omega)
=\sum_{|e|_1=1}\omega(0,e)1_{B}(t_e\omega).
$$
Using ellipticity, this implies that $\mathbb P[B] \in \{0,1\},$ which again by
part {\it \ref{items:nuEquiv}} of this theorem implies (\ref{e7}).

\medskip

\noindent {\it Proof of parts \ref{items:nuUniqInvCont} and \ref{last-part}}. Let $g$ be a bounded
measurable function on $\Omega$. Let $\nu$ be any
invariant  probability measure for the transition kernel $R$ that is absolutely
continuous with respect to $\mathbb P$. By part {\it \ref{items:nuErgod}} and the
ergodic theorem we have that $ P_\nu$-a.s.

$$
\lim_{n\to\infty}\frac{1}{n}\sum_{k=0}^{n-1}g(\omega_k)
=\int g \,  d\nu.
$$
Now, by part $(i)$ of this theorem, the above convergence
is also occurs $P_\mathbb P$-a.s. Hence, we have that

$$
\lim_{n\to\infty}E_0\left[\frac{1}{n}
\sum_{k=0}^{n-1}g(\omega_k)\right]
=\int g \, d\nu.
$$
This proves the uniqueness of $\nu$ and part $(iv)$.
$\Box$
\medskip

\noindent An important generalization of Kozlov's theorem
was obtained by Rassoul-Agha in \cite{RA03}. There, he shows that
under the assumption that the random walk is directionally
transient, the environment satisfies a certain mixing and uniform ellipticity
condition, and if there exists an invariant probability
measure which is absolutely continuous with respect to 
the initial law $\mathbb P$ in certain half-spaces, a conclusion analogous to
Kozlov's theorem holds. 

In \cite{Len-13},  Lenci
generalizes Kozlov's theorem to environments which are not necessarily
elliptic. \noindent Lenci admits the possibility that
the environment is ergodic with respect to some subgroup
$\Gamma$  strictly smaller than $\mathbb Z^d$, which is a stronger condition
 than total ergodicity, and
 which enables him to relax the ellipticity condition. Furthermore, in Bolthausen-Sznitman \cite{BS-02b},
an example of a RWRE which does not satisfy the ellipticity condition \ref{items:Eassumption} and for which there are no invariant probability measures
for the environmental process which are absolutely
continuous with respect to the initial law of the environment
 is presented (see also \cite{RA03}).

\section{The law of large numbers, directional transience and ballisticity}

For the purposes of applying Kozlov's theorem, it would be important
to understand how to reconstruct the random walk from the canonical
environmental process. Now, let us note that if we
 denote by $\Omega_{per}$ the periodic environments
so that

$$
\Omega_{per}:=\{\omega\in\Omega:\omega=t_x\omega\ {\rm for}\ {\rm some}\
x\in \mathbb Z^d, x\ne 0\},
$$
whenever $\omega\in\Omega \backslash \Omega_{per}$ and $\omega'$ is a translation
of $\omega$, this translation is uniquely defined. This observation
would enable us to express the increments of the random walk as a
function of the environmental process whenever the initial
condition is not periodic. Assuming that the initial law $\mathbb P$
of the environment is ergodic, and noting that the set of periodic environments
is invariant under translations, we can see that $\mathbb P[\Omega_{per}]$
equals either $0$ or $1$. Nevertheless, assuming \ref{items:ERGassumption}, may happen that $\mathbb P[\Omega_{per}]=1$,
a situation where a priori we cannot perform this reconstruction (and which is impossible if we assume even \ref{items:IIDassumption}). We will
therefore prove directly the ergodicity of the increments of the random
walk. 

Our first application of Kozlov's theorem will relate the so-called
transient regime with the ballistic one.

\begin{definition} \label{defs:direcTrans} {\bf (Transience in a
given direction)} For $l\in\mathbb S^{d-1}$
define the event
\begin{equation} \label{eqs:AlDef}
A_l:=\big \{\lim_{n\to\infty}X_n\cdot l=\infty \big \}
\end{equation}
of directional transience in direction $l.$
We will call a RWRE  {\it transient in direction $l$} if ${P_0[A_l]=1.}$ 
\end{definition}

\begin{definition} {\bf (Ballisticity in a
given direction)}  \label{defs:ballisticity}
 Let $l\in\mathbb S^d$. We say that
a RWRE is ballistic in direction $l$, if
$P_0$-a.s.

\begin{equation} \label{eqs:ballisticity}
\liminf_{n\to\infty}\frac{X_n\cdot l}{n}>0.
\end{equation}

\end{definition}
\medskip

\noindent We will see in Chapter 3, that in fact the limit
in the left-hand side of (\ref{eqs:ballisticity})
always exists, and is even know to be deterministic
in dimensions $d=2$.

 Let us
 now
consider for each $x\in\mathbb Z^d$ the {\it local drift}
at site $x$ defined as

$$
d(x,\omega):=\sum_{e \in U}\omega(x,e)e=E_{x,\omega}[X_1-X_0].
$$

\noindent We then have the following corollary to Kozlov's theorem.

\begin{corollary}
\label{corollary-tb} Consider a RWRE in an environment with law $\mathbb P$ fulfilling \ref{items:Eassumption} and
\ref{items:ERGassumption}.
 Furthermore, assume that there exists an invariant probability measure 
for the environment seen from the particle, denoted by $\nu$, which is absolutely
continuous with respect to $\mathbb P$. Then a law
of large number is satisfied so   that $P_{0,\mathbb P}$-a.s.

$$
\lim_{n\to\infty}\frac{X_n}{n}=\int d(0,\omega)\nu(d\omega) =:v.
$$
Furthermore, if the walk is transient in a given direction $l$,
it is necessarily ballistic in that direction so that $v\cdot l\ne 0$.
\end{corollary}

\medskip

\begin{proof} We will follow Sabot \cite{Sa-12}. Define
for $n\ge 1$,

$$
\Delta X_n:=X_{n}-X_{n-1}.
$$
This is a process with state space $\mathcal U:=U^{\mathbb N}$.
In a slight abuse of notation to \eqref{eqs:shiftDef}, we define the canonical shift $\theta:\mathcal U\to
\mathcal U$ via

\begin{equation} \label{eqs:thetaDef}
\theta (\Delta X_1,\Delta X_2,\ldots):=(\Delta X_2,\Delta X_3,\ldots).
\end{equation}
Note that the process $(\Delta X_n)_{n\ge 1}$
is stationary under the law $P_{0,\nu}$.
We will show that in fact the transformation $\theta$ is
ergodic with respect to the space 
$(\mathcal U,\mathcal B(\mathcal U), P_{0,\nu})$,
where $\mathcal B(\mathcal U)$ is the Borel $\sigma$-field
of $\mathcal U$. Let $A\in \mathcal B(\mathcal U)$ be
invariant so that $\theta^{-1}(A)=A$ and define

$$
\psi(x,\omega):=P_{x,\omega}[(\Delta X_n ) \in A].
$$
We claim that

$$
(\psi(X_n,\omega))_{n\ge 0}
$$
is a martingale with respect to the canonical filtration on $\mathcal U$ generated by $(X_n)$.
Indeed,

$$
P_{0,\omega}[(\Delta X_m)\in A \, | \, X_0,\ldots , X_n]
=P_{X_n,\omega}[(\Delta X_{m})\in A]=\psi(X_n,\omega).
$$
Therefore, taking the limit when $n\to\infty$, and for any $\omega,$ the martingale convergence theorem yields that
\begin{equation}
\label{limxn}
\lim_{n\to\infty}\psi(0,\bar\omega_n)=\lim_{n\to\infty}\psi(X_n,\omega)
=1_A((\Delta X_n))\qquad P_{0,\omega}-a.s.
\end{equation}
We now have by the ergodic theorem and Kozlov's theorem that 

$$
\lim_{n\to\infty}\frac{1}{n}\sum_{k=0}^n \psi(0,\bar\omega_n)=\int
\psi(0,\omega)\nu(d\omega) \qquad P_{0,\nu}-a.s.
$$
The limit (\ref{limxn}) now implies that

$$
P_{0,\nu}\left[(\Delta X_n) \in A\right]\in \{0,1\},
$$
which gives us the claimed ergodicity. We thus have that

$$
\lim_{n\to\infty}\frac{X_n}{n}=\lim_{n\to\infty}\frac{1}{n}\sum_{k=1}^n\Delta
X_k
=\int d(0,\omega)\nu(d\omega) \qquad P_{0,\nu}-a.s.
$$
By Kozlov's theorem, we can conclude that the above convergence occurs
$P_{0,\mathbb P}$-a.s.
The second claim of the corollary is immediate from
Lemma \ref{kesten-lemma} above.
\end{proof}

\medskip

\noindent Rassoul-Agha in \cite{RA03}, obtains a version
of Corollary \ref{corollary-tb} where transience is
replaced by the so-called Kalikow's condition \cite{Ka-81}, a stronger
mixing assumption than ergodicity is required, but
it is necessary only to assume the exitence of
an invariant probability measure which is absolutely continuous
with respect to the intial law only on apropriate half-spaces.

On the other hand, combining Corollary \ref{corollary-tb} with Theorem \ref{theorem-alili}, we can now easily derive the following result for
the one-dimensional case, originally proved by Solomon \cite{So-75} for
the i.i.d. case and later extended by Alili \cite{Al-99} to the ergodic case.

\begin{theorem}
\label{thm:oneDimCharacBallist} 
Consider a RWRE in dimension $d=1$ in  an  environment
with law $\mathbb P$ fulfilling \ref{items:Eassumption} and \ref{items:ERGassumption}.
Then, there exists a deterministic $v\in\mathbb R^d$ such that

$$
\lim_{n\to\infty}\frac{X_n}{n}=v,\qquad P_0-a.s.
$$
Furthermore,

\begin{enumerate}

\item If 
{\bf (B+)} is satisfied, then

$$
v=\mathbb E\Big[(1-\rho_0)\sum_{j=0}^\infty\prod_{k=0}^{j-1}\rho_{k+1}\Big].
$$

\item If {\bf (B-)} is satisfied, then

$$
v=\mathbb E\Big[(1-\rho_0^{-1})
\sum_{j=0}^\infty\prod_{k=-1}^{-j}\rho_{-k}^{-1}\Big].
$$

\item \label{items:zeroVel} If neither {\bf (B+)} nor {\bf (B-)}
are satisfied, then

$$
v=0.
$$

\end{enumerate}
\end{theorem}
\medskip

Since in case \ref{items:zeroVel} of the above one has that $X_n/n$ converges to $0,$ one immediately 
is led to the question of the typical order of $X_n$ in this case. The answer to this problem (and further interesting
insight) has 
 been obtained by Kesten, Kozlov
and Spitzer \cite{KKS-75}: In fact, there is a direct connection between the exponent $\kappa \in (0,1)$
 characterized
by 
$$
\mathbb E[ \rho^\kappa] =1,
$$
and the typical order of $X_n$ in this case, which is $n^\kappa.$
We refer the reader to \cite{KKS-75} for further details.

In addition,  from the above discussion we see that in dimension $d=1$,
if the family of integer shifts is ergodic with respect to the law $\mathbb P$ of the environment,
the walk being transient to the right or left does not ensure
the existence of an invariant probability measure 
for the environmental process which is absolutely continuous
with respect to $\mathbb P$. Let us give two examples which show
that this situation could also occur for dimensions $d\ge 2$.

\medskip

\begin{example} \label{ex:I} Let $d=2$. Consider a random walk
in an environment $(\omega(x))_{x\in\mathbb Z^2}$ 
of the form $\omega(x):=(\omega(x,e))_{e\in U}$ with
a law $\mathbb P$ such that 
$\mathbb P[\omega(x,e)=1/4]=1$ for  $e= e_2$ and $e= -e_2$ and
$\omega(x,e_1)=q(x)$ while $\omega(x,-e_1)=p(x) = \frac12-q(x)$, with
$\mathbb E[\log (p(x)/q(x))]<0$ and $\mathbb E[p(x)/q(x)]=1$.
Assume also that for every $x\in\mathbb Z^2$,
 $(\omega(x+ne_1))_{n\in\mathbb Z}$ are i.i.d. under $\mathbb P$ while
$$
\mathbb P[\omega(x+e_2)=\omega(x)]=1
$$
In other words, the environment is constant in the direction $e_2$,
but it is i.i.d. in the direction $e_1$, see Figure \ref{fig:exI} also. It is easy to check that
the shifts $(\theta_x)_{x \in \mathbb Z^d}$ form an ergodic family with respect to $\mathbb P.$
Also, the walk is transient in direction $e_1$, but not ballistic
in that direction and there are no invariant probability measures
for the environmental process which are absolutely continuous with respect to $\mathbb P$
(cf.  Corollary \ref{corollary-tb}). 
\end{example}

\begin{figure}[h]
\begin{center}
\huge
\scalebox{0.45}{\includegraphics{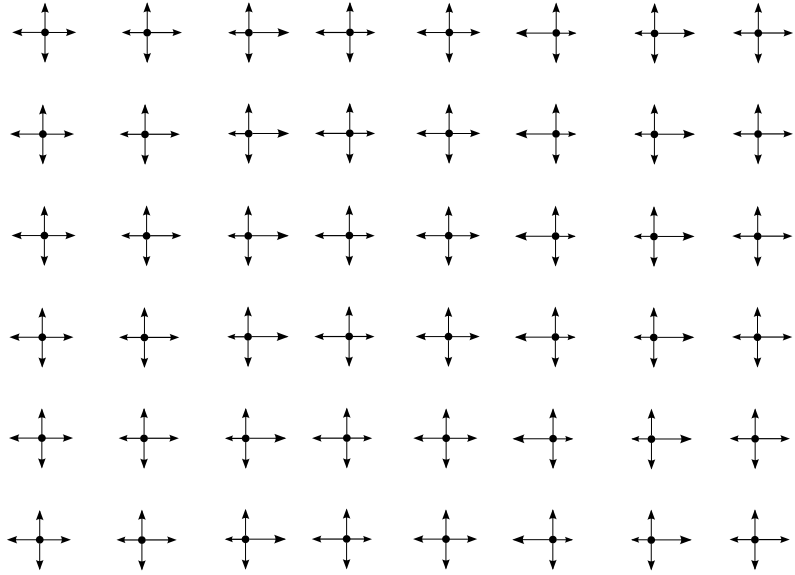}}
\caption{\sl 
A sketch of an environment which is i.i.d. in direction $e_1$ and constant in direction $e_2.$
}
\label{fig:exI}
\end{center}
\end{figure}

\medskip

\begin{example} \label{ex:II} Let $\epsilon>0$.
Furthermore, take $\phi$ to be any random
variable taking values on the interval $(0,1/4)$ and such
that the expected value of $\phi^{-1/2}$ is infinite,
while for every $\epsilon>0$, the expected value of
$\phi^{-(1/2-\epsilon)}$ is finite.
Let $Z$ be a Bernoulli random variable of parameter $1/2$.
We now define $\omega_1(0,e_1)=2\phi$, $\omega_1(0,-e_1)=\phi$,
$\omega_1(0,-e_2)=\phi$ and $\omega_1(0,e_2)=1-4\phi$
and
$\omega_2(0,e_1)=2\phi$, $\omega_2(0,-e_1)=\phi$,
$\omega_2(0,e_2)=\phi$ and $\omega_2(0,-e_2)=1-4\phi$.
We then let the environment at site $0$ be given by
the random variable $\omega(0, \cdot):=Z\omega_1(0, \cdot)
+(1-Z(0))\omega_2(0,\cdot),$ and extend this to an i.i.d. environment on $\mathbb Z^d.$ This environment has the
property that traps can appear, where the random walk
gets caught in an edge, as shown in
Figure \ref{trapInfExp}. Furthermore, as we will
show, it is not
difficult to check that the random walk in this random
environment is transient
in direction $e_1$ but not ballistic. 
Hence, due to Corollary \ref{corollary-tb} there exists no invariant probability measure for the environment seen from the particle, which in addition
is absolutely continuous with respect to $\mathbb P.$
\end{example}

%

\begin{figure}[h]
\begin{center}
\huge

\scalebox{0.650}{\includegraphics{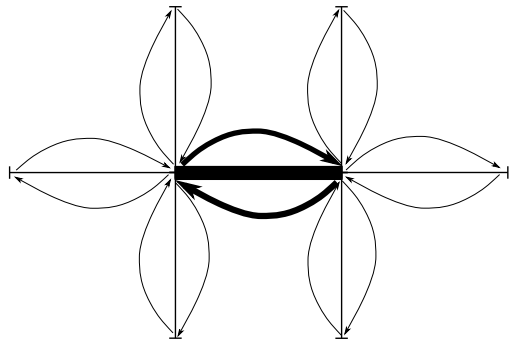}}
\caption{A trap produced by an elliptic environment.}
\label{trapInfExp}
\end{center}
\end{figure} 

\medskip

\noindent These are two examples of walks which are transient
in a given direction but not ballistic, and for which there is no
invariant probability measure for the environmental process absolutely
continuous with respect to the initial law $\mathbb P$
of the environment. It is natural hence to raise the following questions:

\medskip
\begin{question} \label{qn:I} {\it Assume given a RWRE fulfilling \ref{items:ERGassumption} and \ref{items:Eassumption}.
Furthermore assume the RWRE is transient in a given direction.
Is the existence of an invariant probability measure for the environmental
process which is absolutely continuous with respect to $\mathbb P$ equivalent
to ballisticity in the given direction?}
\end{question}


\medskip

\begin{question} \label{qn:transBall} {\it Let $d\ge 2$. Assume given a RWRE for which
\ref{items:UEassumption} and \ref{items:IIDassumption} are fulfilled, and which is transient in
direction $l\in\mathbb S^{d-1}.$ Is the RWRE necessarily ballistic in direction $l$?}
\end{question}

\medskip

\noindent As it is discussed above, example \ref{ex:II} shows that the if
the hypothesis \ref{items:UEassumption} is replaced
 by \ref{items:Eassumption} in the open question \ref{qn:transBall}, then its answer is negative.
The following proposition gives an indication of how much
ellipticity should be required.

\medskip

\begin{proposition}
\label{proposition-ellipticity}
 Consider a random walk in an i.i.d. environment.
Assume that

\begin{equation} \label{eqs:nonIntegrab}
\max_{e\in U}\mathbb E\left[\frac{1}{1-\omega(0,e)\omega(0,-e)}\right]=\infty.
\end{equation}
Then the walk is not ballistic in any direction.
\end{proposition}
\begin{proof}
Fix $e\in U$ and define the first exit time
of the random walk from the edge between $0$ and $e$ as
$$
T_{\{0,e\}}:=\min\left\{n\ge 0: X_n\notin \{0,e\}\right\}.
$$
 We then have
for every $k\ge 0,$ using the notation  
$\omega_1:=\omega(0,e)$ and $\omega_2:=\omega(0,-e)$.
that
\begin{equation}
\nonumber
P_{0,\omega}[T_{\{0,e\}}> 2k]=(\omega_1\omega_2)^k
\end{equation}
and

\begin{equation}
\label{tau1m}
\sum_{k=0}^\infty P_{0,\omega}[T_{\{0,e\}}>2k]=\frac{1}{1-\omega_1\omega_2}.
\end{equation}
Using \eqref{eqs:nonIntegrab}, this implies that
$$
E_0[T_{\{0,e\}}]=\infty.
$$
We can now show using the strong Markov property
under the quenched measure and the i.i.d. nature of
the environment, that for each natural $m>0$, the time
$T_{m}:=\min\{n\ge 0:X_n\cdot l>m\}$ can be bounded from
below by the sum of  a sequence of random variables $\widetilde T_1,\ldots,\widetilde T_m$ which
under the averaged measure are i.i.d. and distributed as $T_{\{0,e\}}$.
This proves that $P_0$-a.s. $T_{m}/m\to\infty$ which
implies that the random walk is not ballistic in direction $l$.
\end{proof}

\medskip

\noindent Based now on Proposition \ref{proposition-ellipticity} 
we have the following extended version of the open question 1.

\medskip

\begin{question} {\it Let $d\ge 2$. Is it the case that every random walk fulfilling 
\ref{items:Eassumption} and \ref{items:IIDassumption}, and satisfying

$$
\max_{e\in U}\mathbb E\left[\frac{1}{1-\omega(0,e)\omega(0,-e)}\right]<\infty,
$$

\noindent and which is transient in
direction $l\in\mathbb S^{d-1},$ is ballistic in direction $l$?}

\end{question}

\medskip

\noindent For the case of an environment fulfilling \ref{items:IIDassumption} and having
a Dirichlet law, the above question was answered positively
by Sabot \cite{Sa-12} in dimensions $d\ge 3$
(see also the  work of Campos and Ram\'\i rez \cite{CR-13}).

\medskip

\section{Transience, recurrence and a quenched invariance principle} 

Similarly to the case of simple random walk, one of the most basic questions for RWRE is a classification in terms 
of transience and recurrence. As simple as this question is to pose, it is still far from being completely understood.
In fact, a natural question is the following one.

\begin{question} Is it the case that
 in dimensions $d \ge 3,$ a RWRE fulfilling \ref{items:Eassumption} and \ref{items:IIDassumption} is transient?
\end{question}

\noindent This question has been answered
 only in the case of the so-called Dirichlet 
environment (see Sabot \cite{Sa-11})
 and essentially also for balanced environments (see Lawler \cite{Law-82}).
It is intimately related to the quenched
central limit theorem. In this section, we will discuss how
Kozlov's theorem can be used for balanced random walks to derive
such a theorem, from which eventually transience in direction $d\ge 3$
can be deduced.

Consider the subset the set of environments

$$
\Omega_0:=\{\omega\in\Omega: \omega(x,e)=\omega(x,-e)
\ {\rm for}\ {\rm all}\ x\in\mathbb Z^d, e\in U\}.
$$
We will say that the law $\mathbb P$ of the environment
of a RWRE is  {\it balanced} if

$$
\mathbb P[\Omega_0]=1,
$$
where in particular we use that $\Omega_0$ is a measurable subset of $\Omega.$
The following result was proved by Lawler in \cite{Law-82}.

\smallskip

\begin{theorem}
\label{lawler} Consider a random walk with an environment
which has a law $\mathbb P$ fulfilling \ref{items:UEassumption} as well as \ref{items:ERGassumption}, and which
is balanced.
 Then there exists an invariant measure
for the environmental process
which is absolutely continuous with respect to $\mathbb P$.
\end{theorem}

\smallskip

\noindent The above result is one of the few instances in which
it has been possible to construct an absolutely continuous invariant measure
for the environmental process in dimensions $d\ge 2$ (for non-nestling
random walks at low disorder Bolthausen and Sznitman also
make such a construction in \cite{BS-02b}; for
random environment with Dirichlet law Sabot characterizes the cases
when this happens in \cite{Sa-12}). As a corollary, Lawler
can prove
the following.

\smallskip

\begin{corollary}
\label{lawler-clt} Under the conditions of Theorem \ref{lawler}
for $\mathbb P$-a.e. $\omega$, under $P_{0,\omega}$,
the sequence $X_{[n\cdot]}/\sqrt{n}$ converges in law
on the Skorokhod space $D([0,\infty);{\mathbb R}^d)$ to a non-degenerate
Brownian motion with a diagonal and deterministic
covariance matrix $A:=\{a_{i,j}\}$, $a_{i,j}=a_i\delta_{i,j}$.
\end{corollary}

\smallskip

\begin{proof} Let us first explain how to prove
the convergence of the finite-dimensional ditributions. Note that for every $\theta\in\mathbb R^d$ sufficiently close to $0$ and using $\mathbb P [\Omega_0]=1$ we have that

$$
e^{iX_n\cdot\theta-\sum_{k=0}^{n-1}\ln
 \left(2\sum_{j=1}^d\cos (e_j\cdot\theta)
\omega(X_k,e_j)\right)}
$$
is a martingale in $n$ with respect to the law $P_{0,\omega}$.
Therefore,  rescaling $\theta$ by $\theta/\sqrt{n}$
we see that for all $n$ large enough,

$$
E_{0,\omega}\left[e^{i\frac{X_n}{{\sqrt{n}}}\cdot\theta-\sum_{k=0}^{n-1}\ln
 \left(2\sum_{j=1}^d\cos \left(\frac{\theta_j}{\sqrt{n}}\right)
\omega(X_k,e_j)\right)}\right]=1.
$$
Hence, it is enough to prove that there exist
constants $\{a_i:1\le i\le d\}$ such that $P_0$-a.s. one has that

\begin{equation}
\label{l1}
\lim_{n\to\infty}\sum_{k=0}^{n-1}\ln\left(2\sum_{j=1}^d\cos \left(\frac{\theta_j}{\sqrt{n}}\right)
\omega(X_k,e_j)\right)=-\sum_{j=1}^d \frac{a_j}{2}\theta_j^2.
\end{equation}
Now, by Taylor's theorem,

$$
\cos (x)=1-\frac{x^2}{2!}+h_1(x)x^2,
$$
where $\lim_{x\to 0}h_1(x)=0$. Hence,

$$
\cos\left(\frac{\theta_j}{\sqrt{n}}\right)
=1-\frac{\theta_j^2}{2n}+
\frac{\theta_j^2}{n}h_1\left(\frac{\theta_j}{\sqrt{n}}\right),
$$
and for each $k\ge 0$,
\begin{equation}
\label{first}
2\sum_{j=1}^d\cos\left(\frac{\theta_j}{\sqrt{n}}\right)\omega(X_k,e_j)
=1-\sum_{j=1}^d\frac{\theta_j^2}{n}\omega(X_k,e_j)
+
2\sum_{j=1}^d\frac{\theta_j^2}{n}h_1\left(\frac{\theta_j}{\sqrt{n}}\right)\omega(X_k,e_j).
\end{equation}
 A second application of
Taylor's theorem gives that

$$
\ln(1-x)=-x+h_2(x)x,
$$
where $\lim_{x\to 0}h_2(x)=0$.
Thus, using (\ref{first}) we have that,

\begin{align*}
\ln\left(2\sum_{j=1}^d\cos\left(\frac{\theta_j}{\sqrt{n}}\right)
\omega(X_k,e_j)\right)
&=-\sum_{j=1}^d
\frac{\theta_j^2}{n}\bar\omega_k(0,e_j)
+
2\sum_{j=1}^d\frac{\theta_j^2}{n}h_1\left(\frac{\theta_j}{\sqrt{n}}\right)
\bar\omega_k(0,e_j)\\
& \quad+
\left(\sum_{j=1}^d
\frac{\theta_j^2}{n}\bar\omega_k(0,e_j)
-
2\sum_{j=1}^d\frac{\theta_j^2}{n}h_1\left(\frac{\theta_j}{\sqrt{n}}\right)
\bar\omega_k(0,e_j)\right)h_2,
\end{align*}
where 
$$
h_2=h_2
\left(\sum_{j=1}^d
\frac{\theta_j^2}{n}\bar\omega_k(0,e_j)
-
2\sum_{j=1}^d\frac{\theta_j^2}{n}h_1\left(\frac{\theta_j}{\sqrt{n}}\right)
\bar\omega_k(0,e_j)\right),
$$
and where we recall that the environmental process $(\bar \omega_n)$
has been introduced in Definition \ref{defs:envProc}.
It then follows that if we are able to prove that for
each $1\le j\le d$,
$P_0$-a.s. one has that

\begin{equation}
\label{third}
\lim_{n\to\infty}\frac{1}{n}\sum_{k=0}^n\bar\omega_k(0,e_j)
=\frac{a_j}{2},
\end{equation}
then we have proven (\ref{l1}).
%
%
%
%
%
To prove (\ref{third}), by Kozlov's theorem, it is enough to
use Theorem \ref{lawler} which ensures the existence of a measure
$\nu$ which is an invariant measure for the
process $(\bar \omega_n)$  and which is absolutely continuous with respect
to $\mathbb P$. The prove the convergence to Brownian motion we
can use the martingale convergence theorem (\cite{Sz-04}).

\end{proof}

\smallskip
We will now explain the main ideas in the proof of Theorem \ref{lawler}.
the details of which can be found for example in Sznitman \cite{BS-02}.
We will construct an invariant measure by
approximating it with invariant measures with respect to the
environmental processes on finite spaces.
Configurations of the environment on these finite spaces will then
 correspond
to periodic configurations on the full space.
The point
is to do this in such a way that the density of these invariant
measures with respect to  periodized versions of the measure $\mathbb P$,
has an $L_p$ norm for some $p>1$, which is uniformly bounded
in the size of the boxes.

We introduce for $x\in\mathbb Z^d$ the equivalence classes
$$
\hat x:=x+(2N+1)\mathbb Z^d\in\mathbb Z^d/((2N+1)\mathbb Z^d).
$$
In addition we define for $\omega\in\Omega_0$ the corresponding
periodized version  $\omega_N$ of $\omega$ so that $\omega_N(y)=\omega(x)$
for $y\in\mathbb Z^d$ and $x \in B_\infty(N)$ such that $\hat y=\hat x,$.
and set
$$
\Omega_N:=\{\omega_N:\omega\in\Omega_0\}.
$$
It is straightforward to see that
the random walk in the
environment $\omega_N$
has an invariant measure of the form

$$
m_N:=\frac{1}{(2N+1)^d}\sum_{x\in B_\infty(N)}\Phi_N(x)\delta_{\hat x},
$$
for some function $\Phi_N$ on $B_\infty(N)$ such that 
$\sum_{x\in B_\infty(N)}\Phi_N=(2N+1)^d$.
%
%
Now define a probability measure on $\Omega_N$ by

$$
\nu_{N}:=\frac{1}{(2N+1)^d}\sum_{x\in  B_\infty(N)}\Phi_N(x)\delta_{t_x\omega_N}.
$$
%
Now introduce the sequence of measures

$$
\mathbb P_N:=\frac{1}{(2N+1)^d}\sum_{x\in B_\infty(N)}\delta_{t_x\omega_N}.
$$
By the multidimensional ergodic theorem (see \cite[Theorem VIII.6.9]{DF-88}), we have that
$$
\lim_{N\to\infty}\mathbb P_N=\mathbb P\qquad\mathbb P-a.s.
$$
Also, one can see that $\nu_N$ is absolutely continuous
with respect to $\mathbb P_N$,

$$
d\nu_N=:f_N \, d\mathbb P_N,
$$
with
%
%
%
\begin{equation*}
\int
f_N^{\frac{d}{d-1}} \, d\mathbb P_N
\le \frac{1}{(2N+1)^d}\sum_{x\in B_\infty(N)}\Phi_N(x)^{\frac{d}{d-1}}.
\end{equation*}
Hence, for every bounded measurable function $g$ on $\Omega$ we have that
$$
\left|\int g \, d\nu_N\right|\le
\left(\int |g|^d \, d\mathbb P_{N}\right)^{\frac{1}{d}}
\left(\int f_{N}^{\frac{d}{d-1}} \, d\mathbb P_{N}\right)^{\frac{d-1}{d}}
\le \Vert g\Vert_{L^d(\mathbb P_N)} \Vert \Phi_N \Vert_{L^{\frac{d}{d-1}}}.
$$
where we write $L^d$ for the corresponding space with respect to the normalized counting measure on $B_\infty(N).$
Now, assume that there is a constant $C$ such that for every $N$,

\begin{equation}
\label{first-est}
 \Vert \Phi_N \Vert_{L^{\frac{d}{d-1}}} \le C.
\end{equation}
Using the compactness of $\Omega$ and Prohorov's theorem,
we can extract a subsequence $\nu_{N_k}$ of $\nu_N$ which
converges weakly to some limit $\nu$ as $k\to\infty$. Then we would obtain that
$$
\left|\int g \, d\nu\right|\le
 C \Vert g\Vert_{L^d(\mathbb P)},
$$
which would prove that $\nu$ is absolutely continuous with respect to $\mathbb P$.
Note also that Kozlov's theorem (Theorem \ref{thm:Kozlov}) ensures that $\nu$ is deterministic.
Let us now prove (\ref{first-est}). For that purpose, suppose that
for every function $h\in L^d(\mathbb P_N)$,

\begin{equation}
\label{second-est}
\sup_{x\in B_\infty(N),\omega_N}\left|E_{x,\omega_N}\left[\sum_{k=0}^\infty
\left(1-\frac{1}{N^2}\right)^k h(X_{k})\right]\right|\le C N^2 \Vert h \Vert_{L^d(\mathbb P_N)}.
\end{equation}
We claim that (\ref{second-est}) implies (\ref{first-est}). 
Indeed, 
\begin{align*}
\Vert\Phi_N\Vert_{L^{\frac{d}{d-1}}}&=\sup_{h:\Vert h\Vert_{L^d}\le 1}
(\Phi_N,h)
=\sup_{h:\Vert h\Vert_{L^d}\le 1}\frac{1}{N^2}
\sum_{k=0}^\infty\left(1-\frac{1}{N^2}\right)^k
\frac{1}{(2N+1)^d}\sum_{x\in B_\infty(N)}\Phi_N(x)h(x)\\
&=\sup_{h:\Vert h\Vert_{L^d}\le 1}\sum_{k=0}^\infty \frac{1}{N^2}\left(1-\frac{1}{N^2}
\right)^k\frac{1}{(2N+1)^d}\sum_{x\in B_\infty(N)}\Phi_N(x) E_{x,\omega_N}[h(X_{k})],
\end{align*}
which would yield \eqref{first-est}.
We now claim that (\ref{second-est}) is a consequence of
the inequality
\begin{equation}
\label{third-est}
\Vert Q_{\omega} f\Vert_\infty\le C N^2 \left(\frac{1}{(2N+1)^d}
\sum_{x\in B_\infty(N)} |f(x)|^d\right)^{\frac{1}{d}},
\end{equation}
where

\begin{equation}
\nonumber
Q_\omega f(x):= E_{x,\omega_N}\left[\sum_{k=0}^{S_N-1}f(X_{k})\right]
\end{equation}
and

$$
S_N:=\inf\{n\ge 0: \vert X_{n}\vert_\infty\ge N\}.
$$
To prove (\ref{second-est}) assuming (\ref{third-est}), define
$\tau_0:=0$ and

$$
\tau_1:=\tau=\inf\{n\ge 0:\vert X_n-X_0\vert_\infty \ge N\},
$$
as well as recursively for $k\ge 1$,
$\tau_{k+1}:=\tau\circ\theta_{\tau_k}+\tau_k$. Then, for each $\rho\in [0,1)$
we have that

\begin{align*}
E_{x,\omega_N}\left[\sum_{k=0}^\infty \rho^k f(X_{k})\right]
&=E_{x,\omega_N}\left[\sum_{m=0}^\infty\sum_{\tau_m\le k<\tau_{m+1}}\rho^k
f(X_{k})\right]\\
&\le
\sum_{m=0}^\infty \sup_{x\in\mathbb Z^d} E_{x,\omega_N}[\rho^\tau]^m
\sup_{x\in\mathbb Z^d}\left|(Q_{t_x\omega}(t_x f))(0)\right|\\
&
\le C N^2\frac{1}{| B_\infty(N)|^{1/d}}\Vert f\Vert_{L^d}\frac{1}{1-\sup_x E_{x,
\omega_N}[\rho^\tau]}.
\end{align*}
Now, for every $K>0$ we have
$
E_{x,\omega_N}[\rho^\tau]\le P_{x,\omega_N}[\tau\le K]
+\rho^KP_{x,\omega_N}[\tau>K].$
But since the random walk $(X_n)_{n\ge 0}$ is a martingale, by
Doob's martingale inequality we have that for every $\lambda>0$,

$$
\lambda N P_{0,t_x\omega}\Big[ \sup_{0 \le k\le K} \vert X_k^i \vert \ge\lambda N \Big]
\le C' K^{1/2},
$$
for some constant $C'>0$.
Choosing $K=CN^2$ for an appropriate constant $C$, we have that
for an appropriate choice of $\lambda$,
$$
P_{x,\omega}[\tau\le K]\le\sum_{i=1}^dP_{0,t_x\omega}
\left[\sup_{0\le k\le K} \vert X_{k} \vert_\infty\ge \lambda N
\right]\le C' \frac{1}{\lambda}C^{1/2} N\le\frac{1}{2}.
$$
To  finish the proof, it remains to establish (\ref{third-est}).
As explained in Sznitman \cite{BS-02}, one can follow the methods developed by Kuo and Trudinger  
\cite{KT-90} to obtain pointwise estimates for 
linear elliptic difference equations with random
coefficients. One uses the fact that
 $u= Q_\omega f$ is a solution of the equation

\begin{eqnarray*}
&(L_\omega u)(x)=-f(x),\qquad {\rm for}\ x\in B_\infty(N),\\
&u(x)=0, \qquad {\rm for}\ x\in\partial  B_\infty(N);
\end{eqnarray*}
where
\begin{equation}
\label{fourth-eq}
(L_\omega g)(x)=\sum_{e \in U}\omega(x,e)(g(x+e)-g(x)).
\end{equation}
and the  so-called
{\it normal mapping} (see \cite{KT-90}) defined 
for $x\in B_\infty(N)$ as
$$
\chi_u(x):=\big \{p\in\mathbb R^d: u(z)\le u(x)+p\cdot (z-x),\ {\rm for}\
z\in  B_\infty(N)\cup\partial B_\infty(N) \big \}
$$
to conclude that

%
%
%
%
%
%

$$
\omega_d\frac{(\max u)^d}{(2N)^d}=
\big | B_2(\max u/(2N))\big|
\le \sum_{x\in B_\infty(N)}|\chi_u(x)|\le
\sum_{x\in B_\infty(N)}\frac{f(x)^d}{\kappa^d},
$$
where $\omega_d$ is the volume of a sphere unit radius,
which proves  (\ref{third-est}).
%

\bigskip

 Theorems \ref{lawler} and  \ref{lawler-clt}
have recently been been extended by Guo and Zeitouni in \cite{GZ-12} to
the elliptic case.
Further progress has been made by Berger and Deuschel in \cite{BD-12}.
They introduce the following concept which is considerably
weaker than ellipticity.

\begin{definition} {\bf (Genuinely $d$-dimensional environment)}
We say that an environment $\omega\in\Omega$ is a
{\it genuinely $d$-dimensional environment} if
for every $e \in U$ there exists a $y\in\mathbb Z^d$
such that $\omega(y,e)>0$. We say that the law $\mathbb P$ of an
environment is
{\it genuinely $d$-dimensional} if
environments are genuinely $d$-dimensional under $\mathbb P$ with probability
one.
\end{definition}

\begin{theorem} (\cite{BD-12}) Consider
a RWRE in an i.i.d., balanced and genuinely $d$-dimensional environment. Then the quenched invariance principle holds with a
deterministic non-degenerate diagonal covariance matrix.
\end{theorem}

\noindent In \cite{Z-04}, Zeitouni proves as a corollary of Lawler's
quenched central limit theorem for balanced random walks the following result.

\smallskip

\begin{theorem} (\cite[Theorem 3.3.22]{Z-04}) Under the conditions of Theorem \ref{lawler}, the random walk
is transient in dimensions $d\ge 3$ and recurrent in dimension $d=2$.
\end{theorem}

\medskip


\section{One-dimensional quenched large deviations}
\label{one}
The following result was first derived by Greven and
den Hollander \cite{GdH-94} to the case of an i.i.d. environment
and then extended by Comets, Gantert and Zeitouni \cite{CGZ-00} for
ergodic environments.

\begin{theorem} {(Greven-den Hollander, Comets-Gantert-Zeitouni)}
\label{gdhcgz}
 Consider a RWRE in dimension $d=1$. 
Assume that $\mathbb E[\log\rho]\le 0$ and that the environment
fulfills \ref{items:Eassumption} and is totally ergodic. Then, there exists
a deterministic rate function $I:\mathbb R\to [0,\infty]$ such that

\begin{enumerate}

\item For every open set $G\subset \mathbb R$ we have
that

$$
\liminf_{n\to\infty}\frac{1}{n} \log P_{0,\omega}\left[\frac{X_n}{n}\in G\right]\ge
-\inf_{x\in G} I(x)\qquad\mathbb P-a.s.
$$

\item For every closed set $C\subset \mathbb R$ we have
that

$$
\limsup_{n\to\infty}\frac{1}{n} \log P_{0,\omega}\left[\frac{X_n}{n}\in C\right]\le
-\inf_{x\in C} I(x)\qquad\mathbb P-a.s.
$$

\end{enumerate}
Furthermore, $I$ is continuous and convex, and it is finite
exactly on $[-1,1]$.
\end{theorem}

\smallskip
\noindent The strategy used by Comets, Gantert and Zeitouni in \cite{CGZ-00}
to prove Theorem 
\ref{gdhcgz} is based on obtaining a
recursion relation for the moment generating function
$\phi(\lambda):=E_{0,\omega}[e^{\lambda T_1}]$, where
for $k\ge 1$,  $T_k:=\inf\{n\ge 0:X_n=k\}$, which leads to a
continuous fraction expansion of it. This leads to
a large deviation principle for $T_k/k$ with
rate function given by the expression

$$
I(t)=\sup_{\lambda\in\mathbb R}(\lambda t- E_0
[\phi(\lambda)]).
$$
As is often the case, the expression for the rate function is much more explicit in $d=1$ than in higher dimensions (cf. also Section \ref{sec:varForm} for the latter).
In addition to the above, in \cite{CGZ-00} the following is also shown.
\begin{theorem} 
\label{slope} Consider a RWRE satisfying the hypotheses of Theorem
\ref{gdhcgz}. Assume that the support of the law of
$\omega(0,1)$ intersects both $\left(0,\frac{1}{2}\right]$ and $\left[\frac{1}{2},1\right)$. Then the rate function $I$
of Theorem \ref{gdhcgz}
 satisfies the following properties

\begin{enumerate}

\item \label{items:rateSym} For $x\in (0,1]$ we have that $I(-x)=I(x)-x \mathbb E [\log\rho]$.

\item  $I(x)=0$ if and only if $x\in [0,v],$ with $v$ denoting the limiting velocity $\lim_{n \to \infty} X_n$ (see also (\ref{eqs:ballLLN}) below).

\end{enumerate}
\end{theorem}

\medskip

Part \ref{items:rateSym} of Theorem \ref{slope} shows that
the slope of the rate function to the left of the
origin does not vanish. A similar phenomenon is
expected to happen for every transient random
walk fulfilling \ref{items:IIDassumption} and \ref{items:UEassumption}
in dimensions $d\ge 2$. This behavior is expected to be
connected to the resolution of a conjecture
about the equivalence of two particular ballisticity
conditions (see \eqref{eqs:TgammaEquiv} below), which will be discussed in
Chapter \ref{chap:II}.

\section{Multidimensional quenched large deviations}
\label{multi}

In \cite{varadhan} Varadhan presented a short proof
of the quenched large deviation
principle  for RWRE in general ergodic environments.
His method is based on the use of the superadditive ergodic theorem.

Note that by the Markov property 
for each environment $\omega$ the $n$-step transition probability
of the random walk (see \eqref{eqs:nstep}) satisfies
for each natural numbers $n$ and $m$ and $x,y\in\mathbb Z^d$
the  inequality

\begin{equation}
\label{submultip}
p^{(n+m)}(0,x+y)\ge p^{(n)}(0,x)p^{(m)}(x,x+y).
\end{equation}
We would like to take logarithms on both sides 
to obtain a superadditive quantity and then apply the subadditive ergodic theorem.
Nevertheless, there are two types of degeneracy that complicate
this operation:

\begin{enumerate} 
 \item 
  $p^{(0)}(x,y,\omega)=0$ for $x\ne y$;
\item \label{items:parity}
$p^{(n)}(x,y,\omega)=0$ whenever $n$ and $|x-y|_1$ do not have the
same parity.
\end{enumerate}

 To avoid them Varadhan introduced the
following {\it smoothed transition probabilities}, defined
for each $c>0$ and $\omega,x,y$ and non-negative real $t$,

$$
q_c(x,y,t):=\sup_{m\ge 0}\{p^{(m)}(x,y,\omega)e^{-c|m-t|}\}.
$$
This regularization method is related to homogenization methods
already developed within the context of the stochastic Hamilton-Jacobi equation
(see for example Kosygina, Rezakhanlou and Varadhan \cite{KRV-06}
and Rezakhanlou \cite{Rk-11}).

\begin{theorem}{(Varadhan)}
\label{varadhan}
 Consider a RWRE 
fulfilling \ref{items:UEassumption} and \ref{items:ERGassumption}.
Then, there exists a  convex rate function $I:\mathbb R\to [0,\infty]$
such that

\begin{enumerate}

\item For every open set $G\subset \mathbb R^d$ we have
that

$$
\liminf_{n\to\infty}\frac{1}{n} \log P_{0,\omega}\left[\frac{X_n}{n}\in G\right]\ge
-\inf_{x\in G} I(x)\qquad\mathbb P-a.s.
$$

\item For every closed set $C\subset \mathbb R^d$ we have
that

$$
\limsup_{n\to\infty}\frac{1}{n} \log P_{0,\omega}\left[\frac{X_n}{n}\in C\right]\le
-\inf_{x\in C} I(x)\qquad\mathbb P-a.s.
$$

\end{enumerate}
Furthermore, $I$ is continuous in $\accentset{\circ} B_1(1)$, lower-semicontinuous
in $B_1(1)$ and $I(x)=\infty$ for $x\notin B_1(1)$.
\end{theorem}

\medskip

\noindent We will present here the proof of Theorem \ref{varadhan} given
by Campos, Drewitz, Rassoul-Agaha, Ram\'\i rez and Sepp\"al\"ainen
in \cite{CDRRS-13} and which is valid also for time-dependent
random environments satisfying certain ergodicity conditions --- we refer the reader to \cite{CDRRS-13} for further details on the time dependent setting.

The idea is to avoid the degeneracy issues discussed related to \ref{items:parity} above,
by considering the random walk at even and odd times separately.

Let us begin modifying our random walk model, admitting
the possibility that the walk does not move after one step,
so that the set of  jumps
after one step is now $U':=U\cup\{0\}$
and

$$
\omega(0,0)\ge\kappa.
$$
We will call this random walk the {\it random walk in random
environment with holding times}. We will denote
by $P^h_{x,\omega}$ its quenched law starting from $x$
and
by
$$
p^{(n)}_h(x,y,\omega):=P_{x,\omega}^h[X_n=y]
$$
its $n$-step transition probabilities.
For $x\in\mathbb R^d$, we will define

$$
[x]:=([x_1],\ldots,[x_d])\in\mathbb Z^d.
$$
Let us define for $n\ge 0$, $R_n$ as the set of sites that the
random walk can visit with positive probability at time $n$. Thus,
$R_0:=\{0\}$, $R_1:=U'$ while for $n\ge 1$,

$$
R_{n+1}:=\{y\in\mathbb Z^d:y=x+e\ {\rm for}\ {\rm some}\ x\in R_n\ 
{\rm and}\ e\in U'\} = R_{n+1} + (R_n + U).
$$
It is easy to check that $B_1(1)$ equals the set of limit points
of the sequence of sets $R_n/n$. Furthermore, 

\begin{equation}
\label{lemma31}
R_n=(n B_1(1))\cap\mathbb Z^d
\end{equation}
(see also Lemma 3.1 in \cite{CDRRS-13}). 
We will now prove the following.

\begin{proposition}
\label{rational-rate} Consider a random walk in random
environment with holding times, and which fulfills \ref{items:UEassumption}  and \ref{items:ERGassumption} to hold. Then, for each $x\in  \mathbb Q^d$ we have that
$\mathbb P$-a.s. the limit

\begin{equation} \label{eqs:rateFunc}
I(x):=-\lim_{n\to\infty}\frac{1}{n}\log p^{(n)}_h(0,[nx])
\end{equation}
exists, is convex and deterministic.  Furthermore, $I(x)<\infty$ 
for $x\in \mathbb Q^d\cap \accentset{\circ}B_1(1)$.
\end{proposition}
\begin{proof}
Note that from Lemma \ref{lemma31} we can check that if $x\notin B_1(1)$, for every $n\ge 1$
one has that $nx\notin nB_1(1)$ so that $nx\notin R_n$, and thus
$p^{(n)}_h(0,[nx])=0$. This proves that $I(x)=\infty$ if $x\notin B_1(1)$.

Let us now consider an $x\in\mathbb Q^d\cap \accentset{\circ} B_1(1) .$
Note that there exists a $k\in\mathbb N$ and
a $y\in\mathbb Z^d\cap k \accentset{\circ}B_1(1)$ such that $x=k^{-1}y;$ in addition, $y\in R_k$.

We will now introduce an auxiliary function $\widetilde I$ and then show that it in fact equals the expression given for $I$ in \eqref{eqs:rateFunc}.
Indeed, by the convexity of $B_1(1)$, the subadditive ergodic theorem
\cite{L85} and \eqref{submultip},
we have that

$$
\widetilde I(k^{-1}y):=-\lim_{m\to\infty}\frac{1}{mk}\log p^{(mk)}_h(0,my)
$$
exists $\mathbb P$-a.s. Furthermore, this definition is
independent of the representation of $x$. Indeed, if
$x=k^{-1}y_1=l^{-1}y_2$ for some $k,l\in\mathbb N$,
$y_1\in\mathbb Z^d\cap k \accentset{\circ}B_1(1)$ and $y_2\in\mathbb Z^d\cap l \accentset{\circ}B_1(1)$,
we have that

$$
\widetilde I(k^{-1}y_1)=-\lim_{n\to\infty}\frac{1}{nlk}\log p^{(nlk)}_h(0,nly_1)
=-\lim_{n\to\infty}\frac{1}{nlk}\log p^{(nlk)}_h(0,nky_2)=\widetilde I(l^{-1}y_2).
$$

We will next prove that $\widetilde I$ is deterministic on $\mathbb Q^d\cap
\mathbb Z^d$. Let $x\in\mathbb Q^d\cap \accentset{\circ} B_1(1)$.  There exists a $k\in
\mathbb N$ and a $y\in\mathbb Z^d\cap k \accentset{\circ} B_1(1)$ such that
$x=k^{-1}y$. Now it is enough to prove
that for each $z\in U$ one has that

$$
\widetilde I(x,\omega)\le\widetilde I(x, t_z\omega)
=-\lim_{m\to\infty}\frac{1}{mk}\log p^{(mk)}_h(z,my+z).
$$
But for each $n\in\mathbb N$, we have that

$$
-\frac{1}{mnk}\log p^{(mnk)}_h(0,mny)\le -\frac{1}{mnk}
\log p^{(mnk)}_h(0,z)-\frac{1}{mnk}\log p^{(mnk)}_h(z,mny).
$$
By uniform ellipticity, the first term in the right-hand
side of the above inequality tends to $0$ as $m\to\infty$.
Therefore,

$$
\widetilde I(x,\omega)=-\lim_{m\to\infty}\frac{1}{mnk}\log p^{(mnk)}_h(0,mny)
\le
-\liminf_{m\to\infty}\frac{1}{mnk}\log p^{(mnk)}_h(z,mny).
$$
On the other hand,
\begin{align*}
&-\frac{1}{mnk}\log p^{(mnk)}_h(z,mny)\\
&\le
-\frac{1}{mnk}\log p^{((m-1)nk)}_h(z,(m-1)ny+z)
-\frac{1}{mnk}\log p^{(nk-1)}_h((m-1)ny+z,mny).
\end{align*}
Now, since $z\in U$, one can check that
$p^{(nk-1)}_h((m-1)ny+z,mny)\ge\kappa^{nk-1}$, so that the
last term of the above inequality tends to $0$ when
$m\to\infty$.  We can then conclude that
$\widetilde I(x,\omega)\le\widetilde I(x,t_z\omega)$.

We will now prove that $I$ is well defined in $\mathbb Q^d\cap \accentset{\circ} B_1(1)$
and that it equals $\widetilde I$ there. Let $x\in\mathbb Q^d\cap \accentset{\circ} B_1(1)$.
Furthermore, choose $k$ such that $kx\in\mathbb Z^d$ and given $n \in \mathbb N$
define 

$$
m:=\left[\frac{n}{k}\right].
$$
Necessarily, we can find a sequence $z_1,\ldots,z_{n-mk}\in U$
such that

$$
[nx]=mkx+z_1+\cdots+z_{n-mk}.
$$
Hence, by superadditivity and uniform ellipticity we have that

$$
-\frac{1}{n}\log p^{(n)}_h(0,[nx])\le
-\frac{1}{n}\log p^{(mk)}_h(0,mkx)-\frac{1}{n}\log\kappa^{n-mk}.
$$
Therefore

$$
-\limsup_{n\to\infty}\frac{1}{n}\log p^{(n)}_h(0,[nx])\le\widetilde I(x).
$$
Using a similar argument we can establish that

$$
-\liminf_{n\to\infty}\frac{1}{n}\log p^{(n)}_h(0,[nx])\ge\widetilde I(x).
$$
\end{proof}

We want now to extend Proposition \ref{rational-rate} to $x\in\mathbb R^d$.
To do this, we will need to establish a lemma which
in some sense shows that the quantity $-\log p^{(n)}(0,[nx])$
is continuous as a function of $x$.
For each $x\in\mathbb Z^d$ we define $s(x)$ as
the minimum number $n$ of steps
required for the random walk to move from $0$ to $x$. so that

$$
s(x):=\min\{n\ge 0: x\in R_n\}.
$$
We will now define a norm in $\mathbb R^d$ as follows.
For each $y\in\partial B_1(1)$ we set $\Vert y\Vert:=1$. Then,
for each $x\in\mathbb R^d$ of the form $x=ay$
for some $a\ge 0$, we define $\Vert x\Vert:=a$. Since $B_1(1)$ is convex,
symmetric (in the sense that $x\in B_1(1)$ implies that $-x\in B_1(1)$),
this implies that this defines a norm. It is easy to check
that for every $x\in\mathbb Z^d$,

\begin{equation}
\label{norm}
\Vert x\Vert\le s(x)\le \Vert x\Vert+1.
\end{equation}

\begin{lemma}
\label{crucial} Let $z\in B_1(1)$ and $x\in \accentset{\circ}B_1(1)$. 

\begin{enumerate}

\item \label{items:n2ex}  For
each natural $n$ there exists an $n_2$ such that

\begin{equation}
\label{en2}
n\le n_2\le n+\frac{4d+1}{1-\Vert x\Vert}+n\frac{\Vert x-z\Vert}{1-\Vert x\Vert}+1.
\end{equation}
and such that 

$$
-\log p^{(n_2)}_h(0,[n_2x])\le -\log p^{(n)}_h(0,[nz])-\log\kappa^{n_2-n}.
$$

\item \label{items:n0ex}
Similarly, whenever $\Vert x-z\Vert<1-\Vert x\Vert$, there  exists an $n_0$ such that
for each natural $n\ge n_0$ there exists an $n_1$ such that

\begin{equation}
\label{len2}
n-\frac{4d+1}{1-\Vert x\Vert}-n\frac{\Vert x-z\Vert}{1-\Vert x \Vert}-1\le n_1\le n
\end{equation}
and such that

$$
-\log p^{(n)}_h(0,[nz])\le -\log p^{(n_1)}_h(0,[n_1x])-\log\kappa^{n-n_1}.
$$
\end{enumerate}
\end{lemma}
\begin{proof} To prove part \ref{items:n2ex} of the lemma, it is enough to
show that there exists an $n_2\ge n$ satisfying (\ref{en2}) and such that

\begin{equation}
\label{sen2}
s([n_2x]-[nz])\le n_2-n.
\end{equation}
But by (\ref{norm}) and the fact that $\Vert x-[x]\Vert\le d$ we see that

\begin{eqnarray*}
& s([n_2x]-[nz])\le \Vert[n_2x]-[nz]\Vert+1\le \Vert[n_2x]-[nx]\Vert+\Vert[nx]-[nz]\Vert+1\\
&\le \Vert(n_2-n)x\Vert+\Vert n(x-z)\Vert+4d+1=(n_2-n)\Vert x\Vert+n\Vert x-z\Vert+4d+1.
\end{eqnarray*}
This shows that (\ref{sen2}) is satisfied whenever

$$
n_2\ge n+\frac{4d+1}{1-\Vert x\Vert}+n\frac{\Vert x-z\Vert}{1-\Vert x\Vert}.
$$
To prove part \ref{items:n0ex} of the lemma, note that it is enough to
show that there exists an $n_1\le n$ satisfying (\ref{len2}) and

$$
s([nz]-[n_1x])\le n-n_1.
$$
But,

$$
s([nz]-[n_1x])\le n\Vert z-x\Vert+(n-n_1)\Vert x\Vert+4d+1
$$
which is equivalent to

$$
n_1\le n-\frac{4d+1}{1-\Vert x\Vert}-n\frac{\Vert z-x\Vert}{1-\Vert x\Vert}.
$$
\end{proof}

\medskip

We are now in a position to extend Proposition \ref{rational-rate} to
the following.

\begin{proposition}
\label{real-rate} Consider a random walk in random
environment with holding times, where the law $\mathbb P$ of
the environment is totally ergodic. Then, for each $x\in \mathbb R^d$ we have that
$\mathbb P$-a.s. the limit

$$
I(x):=-\lim_{n\to\infty}\frac{1}{n}\log p^{(n)}_h(0,[nx])
$$
exists, is convex and deterministic.  Furthermore, $I(x)<\infty$ if
and only if $x\in B_1(1)$.
\end{proposition}
\begin{proof}
Let $z\in\mathbb R^d\cap \accentset{\circ} B_1(1)$. Choose a point
$x$ with rational coordinates such that $\Vert z-x\Vert<1-\Vert x\Vert$
and $\frac{1}{1-\Vert x\Vert}\le 2\frac{1}{1-\Vert z\Vert}$. By Lemma \ref{crucial},
for each $n\ge n_0$ we can find $n_1$ and $n_2$
satisfying (\ref{en2}) and (\ref{len2}) and such that

$$
-\frac{n_2}{n}\frac{1}{n_2}\log p^{(n_2)}_h(0,[n_2x])
\le -\frac{1}{n}\log p^{(n)}_h(0,[nz])+b\left(\frac{n_2}{n}-1\right)
$$
and

$$
-\frac{1}{n}\log p^{(n)}_h(0,[nz])
\le -\frac{n_1}{n}\frac{1}{n_1}\log p^{(n_1)}_h(0,[n_1x])+b\left(1-\frac{n_1}{n}\right),
$$
where $b:=-\log\kappa$. From inequalities (\ref{en2}) and (\ref{len2})
of Lemma \ref{crucial} and by Proposition \ref{rational-rate} we can then conclude that

$$
I(x)\le-\liminf_{n\to\infty}\frac{1}{n}\log p^{(n)}_h(0,[nz])
+C(z)b\Vert x-z\Vert
$$
and 
$$
-\liminf_{n\to\infty}\frac{1}{n}\log p^{(n)}_h(0,[nz])\le I(x)
+C(z)b\Vert x-z\Vert,
$$
where $C(z):=2\frac{1}{1-\Vert z\Vert}$. Letting $x\to z$ we conclude
that $I$ is well defined on $\mathbb R^d\cap \accentset{\circ} B_1(1)$.
\end{proof}

\medskip

We are now in a position to extend the function $I$ of
Proposition \ref{real-rate} from $\accentset{\circ} B_1(1)$ to $B_1(1)$ as

$$
I(x):=\left\{\begin{array}{ccc}
I(x),\qquad &\textrm{if}&\quad x\in \accentset{\circ} B_1(1),\\
\liminf_{\accentset{\circ} B_1(1) \ni y\to x}I(y),\qquad &\textrm{if}&\quad x\in\partial B_1(1).
\end{array} \right.
$$
We will show that this is in fact the rate function of
Theorem \ref{varadhan}, but of a RWRE with holding times.
 Let us first show that $I$
satisfies the requirements of Theorem \ref{varadhan}.
By uniform ellipticity, it is clear that $I(x)\le |\log\kappa|$
whenever $x\in B_1(1)$. Also, the proof of Proposition \ref{real-rate}
shows that $I$ is continuous in $\accentset{\circ} B_1(1)$. Furthermore, it is
obvious that $I$ is convex and lower-semicontinuous
in $B_1(1)$.

Now, note that if $G$ is an open subset of $\mathbb R^d$ 
and $x\in G$, the sequence $[nx]$ is in $nG\cap\mathbb Z^d$
 and

$$
P^h_{0,\omega} \Big[ \frac{X_n}{n}\in G \Big]\ge P^h_{0,\omega}[X_n=[nx]].
$$
In combination with Proposition \ref{real-rate} we therefore conclude that

$$
\liminf_{n\to\infty}\frac{1}{n}\log P^h_{0,\omega}
\Big[ \frac{X_n}{n}\in G\Big ]\ge 
-\inf_{x\in G}I(x).
$$
Let us now consider a compact set $C\subset \accentset{\circ} B_1(1)$. We then have that

$$
\limsup_{n\to\infty}\frac{1}{n}\log P^h_{0,\omega}
\Big[ \frac{X_n}{n}\in C \Big]\le
\limsup_{n\to\infty}\sup_{x\in C}\frac{1}{n}\log p^{(n)}_h(0,[nx])
=\inf_n\sup_{x\in C}\sup_{m\ge n}\frac{1}{m}\log p^{(n)}_h(0,[mx]).
$$
Now, through a contradiction argument and an application
of Lemma \ref{crucial}, one can prove that

$$
\sup_{x\in C}\sup_{m\ge n}\frac{1}{m}\log p^{(n)}_h(0,[mx])
\le-\inf_{x\in C}I(x).
$$
This shows that

\begin{equation}
\label{compacts}
\limsup_{n\to\infty}\frac{1}{n}\log P^h_{0,\omega}
\Big[ \frac{X_n}{n}\in C \Big] \le-\inf_{x\in C}I(x).
\end{equation}
Standard arguments using uniform ellipticity enable
us now to extend (\ref{compacts}) from compact sets to closed sets.

One can now derive Theorem \ref{varadhan} for the plain RWRE
from the RWRE with holding times as follows. Define
the {\it even lattice} as $\mathbb Z^d_{even}:=\{x\in\mathbb Z^d:
\vert x \vert_1\ {\rm is}\ {\rm even}\}$. Using the fact that
since $\mathbb Z^d_{even}$ is a free Abelian group it is
isomorphic to $\mathbb Z^d$, we can apply Proposition \ref{real-rate}
for the RWRE with holding times to deduce an analogous result
for the random walk $Y_n:=X_{2n}$ at even times.
On the other hand, using the equality

$$
P_{0,\omega}\Big [\frac{X_{2n+1}}{2n+1}\in A\Big]=
\sum_{i=1}^{2d}\omega(0,e_i)P_{e_i,\omega}\Big[\frac{X_{2n}}{2n}\in A\Big],
$$
and the asymptotic behavior previously proved at even times,
in combination with the assumption of uniform ellipticity, we can deduce the large deviation principle of Theorem \ref{varadhan}.

\section{Rosenbluth's variational formula for the 
multidimensional quenched rate function} \label{sec:varForm}
The drawback of  Theorem \ref{varadhan} is that it 
gives very little information about the rate function
of the quenched large deviations
of the random walk. A partial remedy to this was
obtained by  Rosenbluth \cite{r06}
in his Ph.D. thesis in $2006$, where he
derived a variational expression for the rate function.
To state Rosenbluth's result, it is
more natural to define the RWRE in an
abstract setting, where we first define
the dynamics of the environmental process. 
In analogy to the set of admissible transition kernels $\mathcal P$ defined in \eqref{eqs:transKernels},
we denote by $\mathcal Q$ the set of measurable functions $q:\Omega \times U \mapsto [0,1]$ such that
$
\sum_{e \in U} f(\omega, e) =1
$
for all $\omega \in \Omega.$
Define the function $p \in \mathcal Q$ via $p(\omega,e) := \omega(0,e),$
corresponding to the transition probabilities of the canonical RWRE.
Let us call $\mathcal D$ the set
of measurable functions $\phi:\Omega\to [0,\infty)$
such that $\int\phi d\mathbb P=1$.

\medskip

\begin{theorem} Assume that \ref{items:ERGassumption} is fulfilled and that there is
an $\alpha>0$ such that 
$$
\max_{e \in U} \int  |\ln p(\omega,e)|^{d+\alpha}\mathbb P(d\omega)<\infty.
$$
Then the RWRE satisfies a large deviation principle with rate
function

$$
I(x):=\sup_{\lambda \in \mathbb R^d}\{\lambda \cdot x-\Lambda(\lambda)\},
$$
where

$$
\Lambda(\lambda):=\sup_{q\in\mathcal Q}\sup_{\phi\in\mathcal D}\inf_h\sum_{e\in U}\int\left(
\lambda \cdot e
-\ln\frac{q(\omega,e)}{p(\omega,e)}+h(\omega)-h(T_e\omega)\right)
q(\omega,e)\phi(\omega)\mathbb P(d\omega).
$$

\end{theorem}

\begin{remark}
 The integrability assumption in the above theorem is fulfilled if \ref{items:UEassumption} holds true, for example.
\end{remark}

\noindent 
Note that using canonical LDP machinery, one can show that it is enough
to prove that
$$
\lim_{n\to\infty}\log E_{P_\omega}\big[e^{\lambda \cdot X_n}\big]=\Lambda(\lambda).
$$
We will just give an idea of the proof of the above theorem
deriving the lower bound in the above limit. 
 In analogy to the definition of $P_\omega$ in \eqref{eqs:environmentProcProb}, given $q\in\mathcal Q$,
we denote by $Q_\omega$ the law of the corresponding Markov chain
$(\bar\omega_n)_{n\ge 0}$ starting from $\omega$. We then have

$$
E_{P_\omega}\left[e^{\lambda \cdot X_n}\right]
=E_{Q_\omega}\left[e^{\lambda \cdot X_n}\frac{dP_\omega}{dQ_\omega}\right]
=E_{Q_\omega}\left[\exp\left\{\lambda \cdot X_n-\sum_{k=0}^{n-1}\ln
\frac{q(t_{X_k}\omega,X_{k+1}-X_k)}{p(t_{X_k}\omega,X_{k+1}-X_k)}
\right\}\right].
$$
By Jensen's inequality it follows that

\begin{equation}
\label{ros1}
\liminf_{n\to\infty}\frac{1}{n}\ln E_{P_\omega}\left[e^{\lambda \cdot X_n}\right]
\ge \lim_{n\to\infty}E_{Q_\omega}\left[\frac{1}{n} \lambda \cdot X_n-
\frac{1}{n}\sum_{k=0}^{n-1} \sum_{e \in U} \ln
\frac{q(t_{X_k}\omega,e)}{p(t_{X_k}\omega,e)}
\right].
\end{equation}
Now note that the expectation of the second term of (\ref{ros1})
can be written as
$$
E_{Q_\omega}\left[\frac{1}{n}\sum_{k=0}^{n-1}\ln
\frac{q(t_{X_k}\omega,e)}{p(t_{X_k}\omega,e)}
\right]
=\frac{1}{n}\sum_{k=0}^{n-1}\sum_{e\in U}\ln
\frac{q(t_{X_k}\omega,e)}{p(t_{X_k}\omega,e)} q(t_{X_k}\omega,e).
$$
Let us now assume that the chain $(\bar\omega_n)_{n\ge 0}$
under $Q_\omega$ has an invariant measure $\nu$
which is absolutely continuous with respect to $\mathbb P$.
Let us call $\phi$ the Radon-Nikodym derivative of
$\nu$ with respect to $\mathbb P$.
By Kozlov's theorem (Theorem \ref{thm:Kozlov}), we know that the measure $\nu$
is such that $Q_{\nu}:=\int Q_\omega \nu(d\omega)$ is
ergodic (with respect to the time shifts). It follows that

$$
\lim_{n\to\infty}\frac{1}{n}\sum_{k=0}^{n-1}\sum_{e\in U}
\ln
\frac{q(t_{X_k}\omega,e)}{p(t_{X_k}\omega,e)} q(t_{X_k}\omega,e)
=\int\sum_{e\in U}\ln\frac{q(\omega,e)}{p(\omega,e)}
q(\omega,e)\phi(\omega)\mathbb P(d\omega)\qquad Q_{\nu}-a.a.\, \omega,
$$
and hence that

$$
\lim_{n\to\infty}E_{Q_\omega}\left[
\frac{1}{n}\sum_{k=0}^{n-1}\sum_{e\in U}\ln
\frac{q(t_{X_k}\omega,e)}{p(t_{X_k}\omega,e)} q(t_{X_k}\omega,e) \right]
=\int\sum_{e\in U}\ln\frac{q(\omega,e)}{p(\omega,e)}
q(\omega,e)\phi(\omega) \mathbb P(d\omega) \qquad \mathbb P-a.a. \, \omega.
$$
On the other hand, by the law of large numbers, we have that
the behavior of the first term on the right-hand side of \eqref{ros1} is characterized by
$$
\lim_{n\to\infty}E_{Q_\omega}\left[\frac{1}{n}\lambda \cdot X_n\right]
=\int \sum_{e\in U} \lambda \cdot e q(\omega,e)\phi(\omega) \mathbb P(d\omega) \qquad \mathbb P-a.a. \, \omega.
$$
It follows that if we call $\mathcal Q_0$ the set of transition
probabilities $q$ for which there is an invariant
measure $\nu_q$ which is absolutely continuous with respect to
$\mathbb P$ (and which is unique, by part \ref{items:nuUniqInvCont} of Kozlov's theorem),
with $\phi_q=\frac{d\nu_q}{d\mathbb P}$
we have by (\ref{ros1}) that

$$
\Lambda(\lambda)\ge \sup_{q\in\mathcal Q_0}
\sum_{e\in U}\int \left((\lambda,e)
-\ln\frac{q(\omega,e)}{p(\omega,e)}\right)q(\omega,e)\phi_q(\omega)\mathbb P(d\omega).
$$
Now note that for $\phi\in\mathcal D$, the following
are equivalent

$$
\phi=\phi_q
$$
and

$$
\inf_h\int\sum_{e\in U}\left(h(\omega)-h(T_e\omega)\right)q(\omega,e)
\phi(\omega) \mathbb P(d\omega) =0.
$$
Similarly,

$$
\phi\ne\phi_q
$$
and

$$
\inf_h\int\sum_{e\in U}\left(h(\omega)-h(T_e\omega)\right)q(\omega,e)
\phi(\omega) \mathbb P(d\omega) =-\infty.
$$
Therefore, we conclude that

\begin{align*}
\sup_{q\in\mathcal Q_0}
\sum_{e\in U} & \int \left(\lambda \cdot e
-\ln\frac{q(\omega,e)}{p(\omega,e)}\right)q(\omega,e)\phi_q(\omega)\mathbb P(d\omega)\\
&=
\sup_{q\in\mathcal Q,\phi\in\mathcal D}\inf_h
\sum_{e\in U}\int \left(\lambda \cdot e
-\ln\frac{q(\omega,e)}{p(\omega,e)}+h(\omega)-h(T_e\omega)\right)q(\omega,e)\phi_q(\omega)\mathbb P(d\omega),
\end{align*}
which finishes the sketch of the proof for the lower bound.

A level 2 large deviation principle version of Rosenbluth's
variational formula was derived by Yilmaz in \cite{Y09}.
Subsequently, a level 3 version was derived by Rassoul-Agha and
Sepp\"al\"ainen in \cite{RAS11}.

\chapter{Trapping and ballistic behavior in higher dimensions} \label{chap:II}

In Chapter \ref{chap:I}
we have already considered some situations in which
 one has been able to obtain information not only on transience and ballisticity, but also on the diffusive behavior of RWRE as
 well as its large deviations; in these situations, this supplied us with
 a rather precise understanding of the asymptotic behavior.
The content of this chapter is a more general analysis of RWRE in terms of 
the coarser scales of (directional) transience
and ballistic behavior.

\section{Directional transience} 
%

As we have seen in Chapter \ref{chap:I},
the question of whether under appropriate conditions a RWRE in dimension $d \ge 3$ is transient, remains
essentially unsolved.
More is known, however, about ``transience in a given direction'' which has been introduced in 
Definition \ref{defs:direcTrans}, and we will see how this concept plays a role in the 
investigation of ballistic behavior of RWRE also.
In fact, some quite challenging questions
concerning RWRE are
related to that notion, too, as we will see in this chapter.

In the following, we will tacitly use for $x \in \mathbb Z^d$ the equivalence of the conditions
\begin{align} \label{eqs:equivAS}
\begin{split}
\text{``}P_x[A_l]&=1\text{''},\\
\text{and}\\
\text{``for }\mathbb P\text{-almost all } \omega &\text{ one has }P_{x,\omega}[A_l]=1\text{''}.
\end{split}
\end{align}
Note that this equivalence is a direct consequence of the definition of the averaged measure 
below \eqref{eqs:averagedMeasureDef}.

The following result has essentially been proven by Kalikow \cite{Ka-81} and has been refined in \cite{SzZe-99, ZeMe-01}.

\begin{lemma} \label{lem:kal01} Consider a RWRE
satisfying \ref{items:Eassumption} and \ref{items:IIDassumption}.
Then for every 
$l\in\mathbb S^{d-1}$ we have that

$$
P_0[A_l\cup A_{-l}]\in\{0,1\}.
$$

\end{lemma}
Of course, the above zero-one law seems incomplete and one would like to have a zero-one law for the event $A_l$ already.
Intriguingly, however, it is still not known if 
such a statement holds in full generality.

\begin{question} \label{qn:zeroOne}
 Consider a RWRE satisfying the assumptions \ref{items:Eassumption}  and \ref{items:IIDassumption}.
Is it true that for every $l\in\mathbb S^{d-1}$ one has
\begin{equation} \label{eqs:full01}
P_0[A_l]\in\{0,1\}?
\end{equation}

\end{question}
As we have seen in Theorem \ref{thm:oneDimCharacterisation}, statement \ref{eqs:full01} holds true for $d=1.$ 
In dimension two, it has been proven to hold true by
Zerner and Merkl \cite{ZeMe-01}. In fact, it is also shown in that source that if one assumes the environment to be stationary and ergodic
with respect to lattice translations only, it can indeed happen that $P_0[A_l] \notin \{0,1\}.$

Apart from leading to interesting problems on its own,
the events $A_l$ also play a key role
in the next section in order to define a renewal structure for RWRE.

\section{Renewal structure} \label{sec:regTimes} 
In order to prove some of the main asymptotic
results for RWRE in the directionally transient regime, we will 
define a renewal structure which will help us to
decompose the RWRE in terms of finite i.i.d. (apart from its initial part; see Corollary \ref{cor:iid}) trajectories.
The first use of this  renewal structure in the
context of RWRE is due to Kesten, Kozlov and Spitzer \cite{KKS-75} in the one-dimensional case,
and it has then been generalized to the higher-dimensional case by Sznitman and Zerner \cite{SzZe-99}.
It can be introduced as follows: given a
direction $l\in\mathbb S^{d-1}$,
it is the first time that
the random walk reaches a new maximum level in direction $l$ and such that after this time it
never goes below this maximum in  direction $l.$
Thus, an easy way to define
the renewal time $\tau_1$ 
is via
\begin{equation} \label{eqs:tau1Def}
\tau_1:=\min \Big\{n\ge 1: \max_{0\le m\le n-1}X_m\cdot l<
X_n\cdot l\le \inf_{m\ge n} X_m\cdot l\Big \}.
\end{equation}
Another way to put it is that $\tau_1$ is the first time that the last exit time from a half space of the form
$
\{x \cdot l < r\},
$
some $r \in \mathbb R,$ coincides with the first entrance time into its complement.

In order to introduce notation which is used in the computations below, we give another
 definition of $\tau_1$ in terms of a 
sequence of stopping times; it is slightly more involved. Consider
$$
H_u^l:=\inf\{n\ge 1:X_n\cdot l > u\}
$$
for $u\in\mathbb R$ as well as
$$
D:=\inf\{n\ge 0:X_n\cdot l<X_0\cdot l\}
$$
which are stopping times with respect to the canonical filtration.
Furthermore, set
$$
S_0:=0,\qquad R_0:=X_0\cdot l.
$$
In a slight abuse of notation and similarly to \eqref{eqs:thetaDef}, we will now use $\theta$ to denote the canonical shift
on $(\mathbb Z^d)^{\mathbb N},$ i.e.,
$$
\theta: (x_0, x_1, x_2, \ldots) \mapsto (x_1, x_2, x_3, \ldots),
$$
and for $n \ge 1,$ we define $\theta_n$ to be the $n$-fold composition of $\theta.$
Using this notation, for $k\ge 1$ we now introduce the stopping times
\begin{align} \label{eqs:renewalStruct}
\begin{split} 	
S_{k}:=&H^l_{R_{k-1}}, \qquad 
D_{k}:=\left\{ \begin{array}{rll}
&D\circ\theta_{S_{k}}+S_{k}, \quad &\text{if } S_k < \infty, \\
&\infty, & \text{otherwise},
\end{array} \right.,
\\ &R_{k}:=
\sup\{X_{m}\cdot l: 0\le m\le D_{k}\}.
\end{split}
\end{align}
We then define
\begin{equation} \label{eqs:Kdef}
K:=\inf\{k\ge 0: S_k<\infty, D_k=\infty\},
\end{equation}
and the first {\it renewal time},
$$
\tau_1:=S_K.
$$
Note that $\tau_1$ is not a stopping time with respect to the canonical filtration anymore, since in order to determine whether $\{S_K = m\}$ occurs
one has to ``see into the future'' of $(X_n)$ after time $m.$
One can then recursively define the sequence
of regeneration times $(\tau_k)_{k \in \mathbb N}$ via
$$
\tau_{k+1} = \tau_1 \circ \theta_{\tau_{k}} + \tau_{k}, \quad k \ge 1,
$$
and set $\tau_0 = 0.$ See Figure \ref{fig:renewals} for an illustration of the above renewal structure.

\begin{remark}
\begin{itemize}
\item
Note here that, although not emphasized explicitly in the notation, the definition of the sequence $(\tau_n)$ depends
on the choice of the direction $l;$ if the very choice of $l$ matters, it will usually be clear
from the context.
\item
If working with directions $l$ having rational coordinates, Definition \ref{eqs:tau1Def} works fine.
However, for general directions $l \in \mathbb S^{d-1}$,
one might under some circumstances
 run into slightly more technical argumentations --- e.g., for guaranteeing that each time a renewal time occurs,
the walker has gained some height bounded away from $0$ in direction $l$ (see for example \cite[(1.63)]{Sz-00});
however, these complications do not pose any serious problems.

Note, on the one hand, that one
 way to avoid this kind of technicalities is to replace $H_{R_{k-1}}^l$ in \eqref{eqs:renewalStruct} 
by $H_{R_{k-1}+a}^l$ for some $a > 0,$ as is done for example in \cite{SzZe-99}. On the other hand, however, 
formulas such as in  Lemma \ref{lem:renewalExp} would result to be more complicated, and therefore we stick to the definition
given above.

\end{itemize}
\end{remark}

\begin{figure}[h]
\begin{center}
\huge

\psfrag{X_0}{$X_0$}
\psfrag{XS1}{$X_{S_1}$}
\psfrag{XD1}{$X_{D_1}$}
\psfrag{XR1}{$X_{R_1}$}
\psfrag{Xtau1}{$X_{\tau_1}$}
\psfrag{Xtau2}{$X_{\tau_2}$}
\psfrag{Xtau3}{$X_{\tau_3}$}
\psfrag{Xtau4}{$X_{\tau_4}$}
\psfrag{l}{$l$}


\scalebox{0.55}{\includegraphics{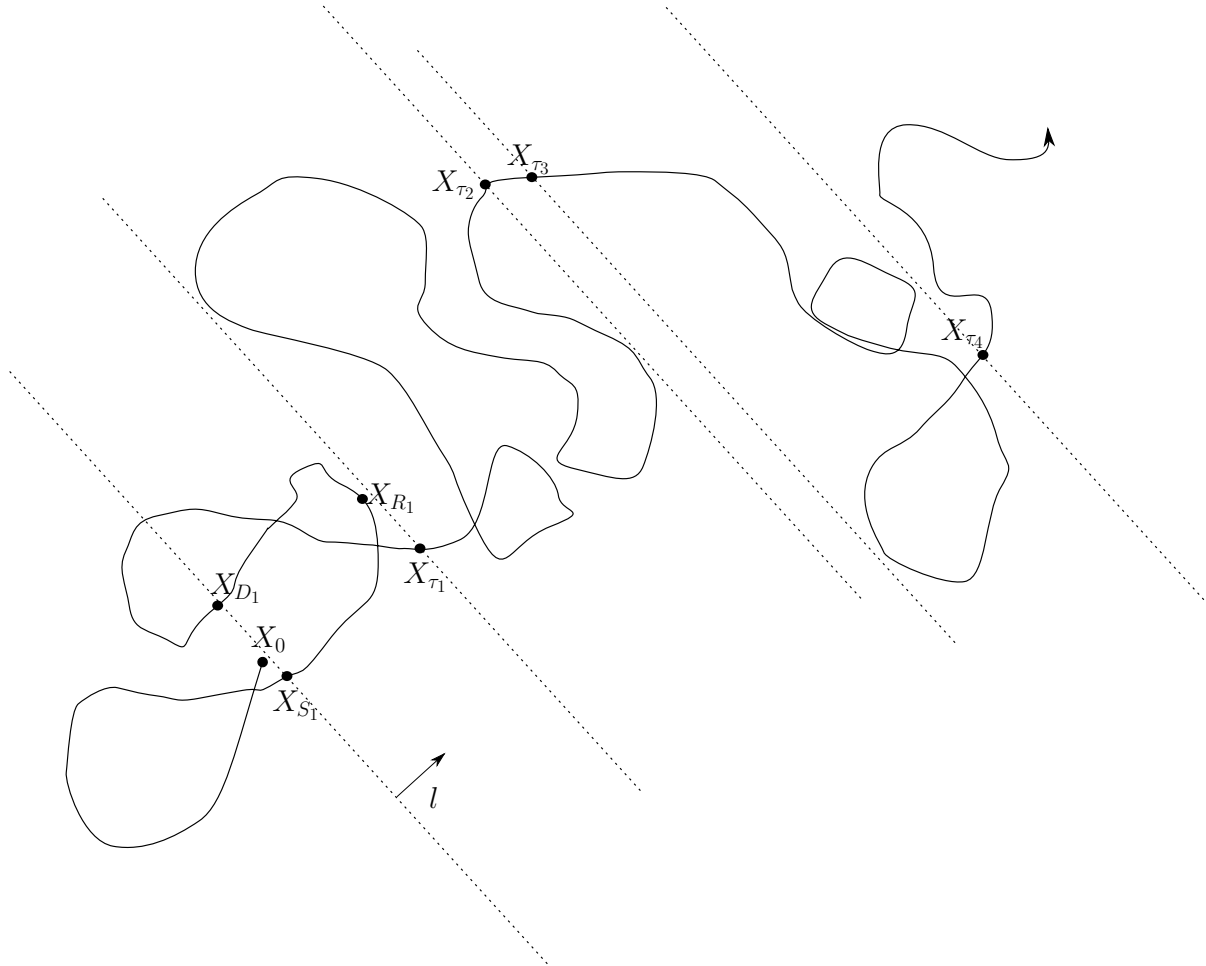}}
\caption{\sl 
Sketch of the renewal structure
}
\label{fig:renewals}
\end{center}
\end{figure}

%
%
%
%
%
%
%
%

The following lemma illustrates the role of the events $A_l$ from \eqref{eqs:AlDef} in the definition 
of the renewal structure described above.

\begin{lemma} \label{lem:transienceAndRegenerations}

Assume \ref{items:Eassumption} and \ref{items:IIDassumption} to hold.
Let furthermore $l\in\mathbb S^{d-1}$ and assume that


\begin{equation} \label{eqs:transPosProb}
P_0[A_l]>0.
\end{equation}
Then the following are satisfied:

\begin{enumerate}

\item \label{items:posEscAtOnce}

\begin{equation} \label{eqs:posEscAtOnce}
P_0[D=\infty]>0;
\end{equation}

\item \label{items:Kfinite}

$$
P_0[A_l \Delta \{K<\infty\}] = 0.
$$
\end{enumerate}

\end{lemma}

In words, Lemma \ref{lem:transienceAndRegenerations} \ref{items:posEscAtOnce}
 states that if the walk has a positive probability of finally escaping
to infinity in direction $l,$ then it must have a positive probability of doing so ``at once'', 
i.e., without entering the half-space $\{x \in \mathbb Z^d \, : \, x \cdot l < 0 \}.$ Part \ref{items:Kfinite} then ensures that on $A_l,$ the above renewal structure
is a.s. well-defined.

\begin{proof}
Let us first prove part \ref{items:posEscAtOnce}. Let \eqref{eqs:transPosProb} be fulfilled and assume that

$$
P_0[D=\infty]=0, \qquad \text{ i.e., } \qquad P_0[D<\infty]=1.
$$
From the invariance of $\mathbb P$ under spatial translations, 
it follows that for all $x\in\mathbb Z^d$ we have
$$
P_x[D<\infty]=1.
$$
Using \eqref{eqs:equivAS}, we deduce that for $\mathbb P$-almost all $\omega$ we would get that for all $x \in \mathbb Z^d,$
$$
P_{x,\omega}[D< \infty] = 1.
$$
Therefore, iteratively applying 
the strong Markov property 
at the return times of the walk to the half-space $\{x \in \mathbb Z^d \, : \, x \cdot l \le 0\},$
 we obtain that
$P_0$-a.s.,
$$
\liminf_{n\to\infty}X_n\cdot l\le 0,
$$
which is a contradiction to \eqref{eqs:transPosProb}.

We now prove part $(ii)$. 
Recalling the definition of $K$ from \eqref{eqs:Kdef}, we note that
$$
\{K<\infty\}\subset A_{-l}^c.
$$
In combination with the zero-one law of Lemma \ref{lem:kal01}, we therefore infer
that
\begin{equation} \label{eqs:symDistOne}
P_0\big[ \{K <\infty\} \backslash A_l \big] = 0.
\end{equation} 
On the other hand, observe that for $k\ge 1$,

\begin{align*}
P_0[R_k <\infty] &=
P_0[S_k<\infty, R_k<\infty] =\mathbb E \big[ E_{0,\omega}[S_k<\infty,P_{X_{S_k},\omega}[D<\infty]] \big]\\
&=\sum_{x\in\mathbb Z^d}\mathbb E[P_{0,\omega}[S_k<\infty, X_{S_k}=x]P_{x,\omega}[D<\infty]]\\
&=
\sum_{x\in\mathbb Z^d}P_0[S_k<\infty, X_{S_k}=x]P_0[D<\infty]\\
&=P_0[S_k<\infty]P_0[D<\infty]
\le
P_0[S_{k-1}<\infty, R_{k-1}<\infty]P_0[D<\infty],
\end{align*}
where to obtain the penultimate equality we used assumption \ref{items:IIDassumption} in combination with the fact that
$
P_{0,\omega}[S_k<\infty, X_{S_k}=x]
$
and 
$
P_{x,\omega}[D<\infty]
$
are measurable with respect to a disjoint set of coordinates in $\Omega.$

It follows that

$$
P_0[R_k<\infty]\le P_0[D<\infty]^k.
$$
Using part \ref{items:posEscAtOnce} of this lemma, this again implies $P_0[K <\infty \, \vert \, A_l]=1,$
which again yields
$$
P_0\big[ A_l \backslash \{K <\infty\}  \big] = 0
$$
and hence  in combination with \eqref{eqs:symDistOne} finishes the proof.

\end{proof}
The next result is contained in \cite[Prop. 1.4]{SzZe-99}
\begin{proposition}
Denote
$$
\mathcal G_1 := \sigma \big( \tau_1, (X_k)_{0 \le k \le \tau_1}, (\omega(y, \cdot))_{\{y \, : \, y\cdot l < X_{\tau_1} \cdot l\}} \big).
$$
Then the joint distribution of
$$
\big( ( X_n -X_{\tau_1})_{n \ge \tau_1}, (\omega(y, \cdot))_{y \cdot l \ge X_{\tau_1} \cdot l} \big)
$$
under
$
P_0[\, \cdot \, \vert \, A_l, \mathcal G_1]
$
equals the joint distribution of
$$
\big( (X_n)_{n \ge 0}, (\omega(y, \cdot))_{y \cdot l \ge 0} \big)
$$
under
$
P_0[\, \cdot \, \vert \, D=\infty].
$
\end{proposition}

In particular, one can infer inductively that on $A_l,$ the sequence of renewal times
$(\tau_n)$ is well-defined.

As a corollary of a slight generalization of the above result, Sznitman and Zerner \cite{SzZe-99} obtain the following.
\begin{corollary}
 \label{cor:iid}
Under $P_0[\, \cdot \, \vert \, A_l],$ the variables $(X_{\tau_{k}} - X_{\tau_{k-1}}, \tau_k)_{k \ge 1}$
 are an independent family. Furthermore, $(X_{\tau_{k}} - X_{\tau_{k-1}}, \tau_k)_{k \ge 2}$,
under $P_0[\, \cdot \, \vert \, A_l]$ are identically distributed 
as $(X_{\tau_1} - X_0,\tau_1)$ under $P_0[\, \cdot \, \vert \, A_l].$
\end{corollary}

On an intuitive level, the idea behind the proof
of Corollary \ref{cor:iid} is that the environments that the walk sees between different renewal times are i.i.d., which
can then be transferred to the behavior of the walk itself.

\section{A general law of large numbers}
Recall that we have already seen a law of large numbers in Corollary \ref{corollary-tb}; however, the assumptions for that
result included the existence of an invariant measure $\nu$ for the environmental process such that 
$\nu$ was absolutely continuous with respect to $\mathbb P.$ We have seen that in some special cases (cf. e.g. Theorem \ref{lawler}),
one can ensure the existence of such a measure $\nu.$
On the other hand, however, not much is known about when such $\nu$ exists, and it would be desirable to have a law of large
numbers that holds without this assumption.

The following theorem is such a result and constitutes a slight refinement of the directional laws of large numbers by 
 Zerner \cite[Theorem 1]{Ze-02} and Zeitouni \cite[Theorem 3.2.2]{Z-04}.

\begin{theorem} \label{thm:LLN}
 Assume \ref{items:IIDassumption} and \ref{items:UEassumption} to hold.
 Then in dimensions $d\ge 2$, there exists a direction
$\nu\in\mathbb S^{d-1},$ and
$v_1,v_2\in [0,1]$ (all deterministic) such that $P_0$-a.s.

\begin{equation} \label{eqs:LLN}
\lim_{n\to\infty}\frac{X_n}{n}=v_1\nu 1_{A_\nu}-v_2\nu 1_{A_{-\nu}}.
\end{equation}
\end{theorem}
\noindent 
\begin{remark}
Let us remark here that on the level of the law of large numbers (in contrast to the central limit theorem or large
deviation results), the averaged result directly implies the $\mathbb P$-a.s. quenched result due to \eqref{eqs:equivAS}.
\end{remark}

Since the conjectured zero-one law 
of open question \ref{qn:zeroOne} is still eluding its complete resolution, the right-hand side of \eqref{eqs:LLN}
might be a non-degenerate random variable.
In dimensions larger or equal to five, Berger \cite{Be-08} has shown that at least  one of the velocities
$v_1$ and $v_2$ must vanish.
In dimension two, the zero-one law of Zerner and Merkl \cite{ZeMe-01} mentioned after
open question \ref{qn:zeroOne} leads to the following corollary of Theorem \ref{thm:LLN}.

\begin{corollary}
 Assume \ref{items:IIDassumption} and \ref{items:UEassumption} to hold.
 Then in dimension $d = 2$, there exists a direction
$\nu\in\mathbb S^{d-1},$ and
$v_1 \in [0,1]$ (all deterministic) such that $P_0$-a.s.

\begin{equation*}
\lim_{n\to\infty}\frac{X_n}{n}=v_1\nu.
\end{equation*}
\end{corollary}

To prove Theorem \ref{thm:LLN}, we need the following lemma.
\begin{lemma} \label{lem:renewalExp}
Assume \ref{items:IIDassumption} and \ref{items:Eassumption} to be fulfilled.
Then for $l =(l_1, \ldots, l_d) \in \mathbb Z^d$ such that ${\rm gcd}(l_1, \ldots, l_d) = 1,$ one has
\begin{equation} \label{eqs:expValFin}
E_0[X_{\tau_1}\cdot l \, | \, D=\infty]
= \frac{1}{P_0[ D = \infty \, \vert \, A_l] \lim_{i \to \infty} P_0 \big[H^l_{i-1} < \infty, X_{H^l_{i}} \cdot l = i\big]}
< \infty.
\end{equation}
(Note that in a slight abuse of notation we use $l \in \mathbb Z^d$ instead of $l \in \mathbb S^{d-1}$ here.)
\end{lemma}
In the case $l=(1, 0, \ldots, 0),$ the
 proof of Lemma \ref{lem:renewalExp} can be found in \cite[Lemma 3.2.5]{TaZe-04} and is based on an argument
by Zerner. See \cite[Lemma 2.5]{DrRa-10a}
for how (in the context of a different renewal structure)
 the generalization to $l$ as in Lemma \ref{lem:renewalExp} works  and how to obtain the finiteness
 of \eqref{eqs:expValFin}.

\begin{proof}[{\scshape Proof of Theorem \ref{thm:LLN}}]
The proof is split into several pieces.

\begin{enumerate}
 \item \label{items:proofPartI}
We  start with proving the following version of a directional law of large numbers, which can be found in 
\cite[Theorem 3.2.2]{TaZe-04}.
It states that for $l \in \mathbb{S}^{d-1}$ with
\begin{equation} \label{eqs:unionTransOne}
P_0[A_l \cup A_{-l}] = 1
\end{equation}
there exist $v_{l}, v_{-l} \in [0,1]$ such that $P_0$-a.s.
\begin{equation} \label{directionalLLN}
\lim_{n \to \infty} \frac{X_n \cdot l}{n} = v_l 1_{A_l} - v_{-l} 1_{A_{-l}}.
\end{equation}

We will prove this result here for $l \in \mathbb Z^d,$ which is slightly easier notationwise.
 Without loss of generality, assume that $P_0[A_l] > 0.$
Then, by the standard law of large numbers in combination with Corollary
\ref{cor:iid}, $P_0[\, \cdot \, \vert \, A_l]$-a.s. we have that
$$
\lim_{k\to\infty}\frac{\tau_k}{k}=E_0[\tau_1 \, | \, D=\infty],
$$
and
$$
\lim_{k\to\infty}\frac{X_{\tau_k}\cdot l}{k}=E_0[X_{\tau_1}\cdot l \, | \, D=\infty].
$$
From this we conclude that $P_0[\, \cdot \, \vert \, A_l]$-a.s.
\begin{equation} \label{eqs:subseqLim}
\lim_{k\to\infty}\frac{X_{\tau_k}\cdot l}{\tau_k}
=\frac{E_0[X_{\tau_1}\cdot l \, | \, D=\infty]}{E_0[\tau_1 \, | \, D=\infty]}=:v_l,
\end{equation}
which due to Lemma \ref{lem:renewalExp} is a finite quantity.
Using the fact that the $\tau_k$ and $X_{\tau_k} \cdot l$ are increasing in $k,$ one obtains the sandwiching
$$
\frac{X_{\tau_k}\cdot l}{\tau_{k+1}} \le \frac{X_n \cdot l }{n} \le \frac{X_{\tau_{k+1}} \cdot l}{\tau_k}
$$
for $\tau_k \le n < \tau_{k+1}.$ 
In combination with \eqref{eqs:subseqLim} we infer that
$$
\lim_{n\to\infty}\frac{X_n\cdot l}{n}=v_l
$$
$P_0[\, \cdot \, \vert \, A_l]$-a.s.
By exchanging $l$ for $-l$ in the above, in combination with \eqref{eqs:unionTransOne}
we therefore obtain \eqref{directionalLLN}.

\item

Next, we will use \cite[Theorem 1]{Ze-02} which states that assuming \ref{items:IIDassumption}, \ref{items:Eassumption}
and $P_0 [A_e \cup A_{-e}] = 0,$ one has for any $e \in U,$ that
\begin{equation} \label{eqs:recLLN}
 \lim_{n \to \infty} \frac{X_n \cdot e}{n} = 0, \quad P_0-a.s.
\end{equation}
On a very coarse heuristic level, the proof of that result is as follows by contradiction: Let 
\begin{equation} \label{eqs:recurrence}
 P_0[A_l \cup A_{-l}]=0,
\end{equation}
and assume  that
$$
\limsup_{n \to \infty} \frac{X_n \cdot l}{n} > 0
$$
with positive probability. Then, if one partitions $\mathbb Z^d$ into slabs orthogonal to $l$ which are
of positive finite thickness, there exists a constant $C$ such that with positive probability, the walk visits
each of a positive fraction of the slabs for at most $C$ time steps. One can next deduce that, denoting the first entrance
position of the walk in such a slab by $x$, there exists a positive number $r$ and a vector $z$ such that with positive probability,
the walk visits the slab for the last time at its $r$-th visit to $x+z.$ From this one is then able to deduce that
one must have $P_0[A_l] > 0,$ a contradiction to \eqref{eqs:recurrence}.
We refer the reader to \cite{Ze-02} for more details.

An inspection of the proof in \cite{Ze-02} yields that by slightly modifying it, one obtains \eqref{eqs:recLLN}
for $e$ replaced by arbitrary $l \in \mathbb S^{d-1}.$
In combination with the result of \eqref{directionalLLN}, and
due to the zero-one law of   Lemma \ref{lem:kal01},
we may therefore omit assumption \eqref{eqs:unionTransOne} and still obtain
that
\eqref{directionalLLN} holds true.

\item

Using \eqref{directionalLLN}, we
obtain that $\lim_{n \to \infty} X_n/n$ exists $P_0$-a.s. and, also $P_0$-a.s, takes values in a set of cardinality
at most $2^d.$
One can then take advantage of  similar arguments as Goergen on page 1112 of \cite{Go-06} in order to show
that
$P_0$-a.s. $\lim_{n \to \infty}X_n/n$ takes values in a set of two elements which are collinear, which finishes
the proof.
Indeed, assume there were $v_1,v_2$ not collinear such that $P_0[\lim_{n \to \infty} X_n/n = v_i] > 0$ for $i =1, 2.$
Then for any $l$ such that 
\begin{equation} \label{eqs:lCond}
 l \cdot v_1, l \cdot v_2 > 0
\end{equation}
 one obtains
by \eqref{directionalLLN} and the fact that
$$
\Big\{ \lim_{n \to \infty} X_n/n = v_1 \Big\} \cup \Big\{ \lim_{n \to \infty} X_n/n = v_2 \Big\}
\subset A_l,
$$
that 
\begin{equation} \label{eqs:projEq}
l\cdot v_1= v_l = l \cdot v_2.
\end{equation}
Since the set of vectors $l$ fulfilling \eqref{eqs:lCond} is open, we can let $l$ vary along a set of basis
vectors fulfilling \eqref{eqs:lCond} and hence conclude that 
\eqref{eqs:projEq} holds for a set of vectors $l$ which form a basis. This implies $v_1 = v_2,$ a contradiction to the assumption
that $v_1$ and $v_2$ were collinear. This yields Theorem \ref{thm:LLN}.

\end{enumerate}
\end{proof}

\begin{remark} \label{rem:nonTrivVel}
 It is useful to observe from part \ref{items:proofPartI} of the proof of
 Lemma \ref{lem:renewalExp} that
 \begin{equation} \label{eqs:nonTrivVel}
  v_l \ne 0  \text{ if and  only if }
E_0[\tau_1 \, | \, D=\infty] < \infty.
\end{equation}
This condition is in general hard to check --- it will be one of the principal goals of the remaining part of these
notes to investigate conditions that ensure $v_l \ne 0$.
\end{remark}

\section{Ballisticity}
We have seen in Theorem \ref{thm:LLN} that a version of a law of large numbers is valid. This, however, did not tell us
anything practical
about the fundamental question
of whether $v_1$ and $v_2$ are equal to or different from $0$ (except for the one-dimensional setting
of Theorem \ref{thm:oneDimCharacBallist}, Remark \ref{rem:nonTrivVel}, and
the result of Berger \cite{Be-08}
alluded to above).
Here, we will address this question and for this purpose recall the concept of ballisticity in a given direction
(see Definition \ref{defs:ballisticity}).

\begin{remark}
If a RWRE is ballistic in a direction $l$ according to Definition \ref{defs:ballisticity}, then one can deduce 
that $P_0$-a.s., the limit
\begin{equation} \label{eqs:limitExistence}
\lim_{n \to \infty} \frac{X_n \cdot l}{n}
\quad \text{exists, is positive, and is $P_0$-a.s. constant.}
\end{equation}
Indeed, if \eqref{eqs:ballisticity} is fulfilled, then $P_0[A_l]=1$ and hence
the renewal structure as introduced in Section \ref{sec:regTimes} is $P_0$-a.s. well-defined
(cf. Lemma \ref{lem:transienceAndRegenerations}).
Similarly to the proof of
Theorem \ref{thm:LLN} one obtains
 that $P_0$-a.s.,
\begin{equation} \label{eqs:balLimitVel}
\lim_{k \to \infty} \frac{X_{\tau_k} \cdot l}{\tau_k} = 
\frac{E_0 [X_{\tau_1} \cdot l\, \vert \, D= \infty]}{E_0 [\tau_1 \, \vert \, D= \infty]}
\end{equation}
exists;  
using \eqref{eqs:ballisticity} we then infer that 
the expression in \eqref{eqs:balLimitVel} must be positive,
which implies \eqref{eqs:limitExistence}.
\end{remark}

If one wants to investigate the occurrence of ballistic behavior in higher dimensions, it is obvious that one cannot expect as simple 
conditions as in the one-dimensional case (cf. Theorem \ref{thm:oneDimCharacBallist})
As a partial remedy, Sznitman \cite{Sz-02} has introduced conditions which in some sense can be considered a
higher-dimensional analog to the conditions given in Theorem \ref{thm:oneDimCharacBallist} for dimension one.
These conditions have turned out to be useful in a plethora of different contexts of RWRE.

\begin{definition} \label{defs:T} {\bf (Conditions $(T)_\gamma$, $(T')$ and $(T)$)}.
Assume $l\in\mathbb S^{d-1}$ and $\gamma\in (0,1]$.
We say that condition $(T)_\gamma|l$ is satisfied if there
exists a neighborhood $V_l$ of $l$
such that for every $l'\in V_l$ one has that

\begin{equation} \label{eqs:condT}
\limsup_{L\to\infty}\frac{1}{L^\gamma}\log P_0 \big[H^{-l'}_{bL} < H^{l'}_L\big]
< 0.
\end{equation}
We say that condition $(T)|l$ is satisfied if condition $(T)_1|l$ holds.
Finally, we say that condition $(T')|l$ is satisfied if for
every $\gamma\in (0,1),$ condition $(T)_\gamma|l$ is satisfied.
Also, if the precise value of $l$ is irrelevant, then we often write $(T)_\gamma$ instead of $(T)_\gamma \vert l,$
and analogously for the remaining conditions.
\end{definition}

Intuitively, if the walk escapes in direction $l'$ and is ``well-behaved'', then the probability in
\eqref{eqs:condT} corresponds to that of a rare event and, due to the independence structure of the environment,
should decay reasonably fast.

\begin{example}
 Zerner and Sznitman \cite{Ze-98,Sz-00} have introduced a classification of RWREs in terms of the
support of the law of the random variable
 \begin{equation}
  d(0,\omega) = \sum_{e \in U}  \omega(0,e) \cdot e;
 \end{equation}
 The random variable $d(0,\cdot)$ is 
the local drift at the origin. Denote by $C \subset D$ (cf. \eqref{eqs:ball}) the convex hull of the
 support of the law of $d(0,\omega).$
 An RWRE is called
 \begin{enumerate}
  \item {\em non-nestling} if
  $$
  0 \notin C;
  $$
  
  \item {\em marginally nestling} if
  $$
  0 \in \partial C;
  $$
  
  \item {\em plain nestling} if
  $$
  0 \in \accentset{\circ} C.
  $$
   
 \end{enumerate}

 In terms of investigating their ballistic behavior, the non-nestling and marginally nestling RWREs are easier to handle than the nestling
 ones. This is due to the fact that  their behavior ``dominates'' that of i.i.d. variables with positive expectation. We leave it to the reader to prove
 that 
 non-nestling RWRE satisfy condition $(T).$

\end{example}

For future purposes it will be helpful to also consider the corresponding polynomial analogues.
\begin{definition} \label{defs:Pasymp}  {\bf (Conditions $\Pasymp_M,$ $\Pasymp_0$)}.
Assume $M > 0$ and $l \in \mathbb S^{d-1}$ to be given. We say that condition $\Pasymp_M \vert l$ (sometimes referred to 
as $\Pasymp_M$ or $\Pasymp$  also) is fulfilled, if there exists a neighborhood $V_l$ of $l$ such 
that for all $l' \in V_l$ and for all
$b > 0$ we have
\begin{equation}
\lim_{L \to \infty} L^{M} P_0 \big[ H^{-l'}_{bL} < H^{l'}_{L} \big]  = 0.
\end{equation}
In addition, we define $\Pasymp_0$ to hold if for all 
$l'$ in a neighborhood of $l$ and for all
$b > 0$ we have
\begin{equation}
\lim_{L \to \infty} P_0 \big[ H^{-l'}_{bL} < H^{l'}_{L} \big]  = 0.
\end{equation}
\end{definition}

\begin{remark}
\begin{itemize}
 \item 
 In the following we will give some fundamental results that were mostly proven under the
 assumption of condition $(T')$.
 However, in anticipation of Theorem \ref{thm:polynomial} below, we will instead formulate them assuming
 $\Pbox_M$ for $M > 15d +5$ only.
 
 \item
 Also, note that due to Theorem \ref{thm:polynomial} it is actually sufficient to assume $\Pbox_M$ (see
 Definition \ref{defs:Pbox})
 instead
 of $\Pasymp_M,$ both for $M > 15d +5,$
 in what follows. This condition is a priori weaker and  has the advantage that it can be checked on finite boxes
 already. 
 However, since it is more complicated to state and needs notation introduced only later on, we will 
 not give its exact definition here yet.
 \end{itemize}
\end{remark}

There is an alternative formulation for the conditions $(T)_\gamma,$ which instead of considering slab exit estimates
involves transience and the (stretched) exponential integrability of the renewal radii.

\begin{theorem} \label{thm:alternativeT} (\cite[Cor. 1.5]{Sz-02})
Assume \ref{items:IIDassumption} and \ref{items:UEassumption} to hold, and 
let furthermore $d\ge 1$ and $\gamma \in (0,1]$. Then the following are
equivalent. 

\begin{enumerate}

\item Condition $(T)_\gamma|l$ is satisfied.

\item \label{items:integrationRenewalR} One has $P_0 [A_l] = 1$ (note that this ensures that $\tau_1$ is well-defined)
and there exists a constant $C>0$ such that

\begin{equation} \label{eqs:expInt}
E_0\Big [\exp \Big\{C^{-1}\max_{0\le i\le \tau_1}|X_i|_1^\gamma \Big\} \Big]<\infty.
\end{equation}

\end{enumerate} 
\end{theorem}

Note
that the first part of the condition \ref{items:integrationRenewalR} in Theorem \ref{thm:alternativeT}
in combination with the law of large numbers of Theorem \ref{thm:LLN} already supplies us with the fact
that $P_0$-a.s., $\lim_n X_n/n$ converges to a deterministic vector. Therefore, due to Theorem \ref{thm:TgammaBallBeh}
below, the
second part of condition \ref{items:integrationRenewalR} in Theorem \ref{thm:alternativeT}
can be seen as guaranteeing that this deterministic limit is different from $0.$ Note, however, that an affirmative
answer to
the open question \ref{qn:transBall}
would imply that the transience assumption $P_0[A_l]$ is already sufficient and the integrability condition of 
\eqref{eqs:expInt} not needed for having a non-zero limiting velocity,
i.e., ballisticity.

These stretched exponential integrability assumptions on the renewal radii have been used by Sznitman (see \cite{Sz-02})
to deduce the following:
In dimensions larger than or equal to two, 
$(T')$ implies a law of large numbers with non-zero limiting velocity as well as an invariance principle for the
RWRE, so that diffusively rescaled it
converges to Brownian motion under the averaged measure.

\begin{theorem} (\cite[Thm. 3.3]{Sz-02}) \label{thm:TgammaBallBeh}
Assume \ref{items:IIDassumption} and \ref{items:UEassumption} to hold. Furthermore, assume $d \ge 2$ and let
 $\Pbox_M \vert l$ is fulfilled for some $l \in \mathbb S^{d-1}$ and $M > 15d+5.$
Then:
\begin{enumerate}
\item \label{items:ballLLN}
The RWRE is ballistic, i.e., one has $P_0$-a.s. that
\begin{equation} \label{eqs:ballLLN}
\lim_{n \to \infty} \frac{X_n}{n} = v \not= 0,
\end{equation}
and where $v$ is deterministic.
\item \label{items:averagedDonsker}
Under $P_0$ and with 
$$
B_t^n := \frac{1}{\sqrt{n}} (X_{\lfloor n t \rfloor} - \lfloor n t \rfloor v), \quad t \ge 0, 
$$
the sequence of processes $((B_t^n)_{t \ge 0})_{n \in \mathbb N}$
converges in law
on the Skorokhod space $D([0,\infty);{\mathbb R}^d)$
to Brownian motion with non-degenerate covariance matrix as $n \to \infty.$
\end{enumerate}
\end{theorem}
\begin{remark}
Recently, there has also been initiated the investigation of ballisticity and related topics for the situation where \ref{items:IIDassumption} holds, but  the
 condition \ref{items:UEassumption} has been replaced by the weaker \ref{items:Eassumption}.
 In this context, in order to obtain results comparable to the ones above,
 one then has to make assumptions on the decay of the random variables $\omega(0,e)$ at $0.$
 These assumptions can be used to apply large deviations estimates in order 
 to obtain that with high $\mathbb P$-probability,
 for sufficiently long paths, the probability of following them is comparable at least to 
 a situation where one has uniform ellipticity; see Section \ref{sec:ellCond} as well as Campos and Ram\'irez
 \cite{CR-13} for further details.
\end{remark}
\begin{question}
Theorem \ref{thm:TgammaBallBeh} states that the conditions $\Pbox_M$ for $M > 15d+5$ do imply a ballistic behavior.
Vice versa, one can ask if \eqref{eqs:ballLLN} already implies the validity of $\Pbox_M$ for $M > 15d+5.$
This question is intimately linked to the slope of the large deviation principle rate function in the origin.
\end{question}

As observed in Remark \ref{rem:nonTrivVel},
in order to guarantee a positive limiting velocity, and therefore to prove 
Theorem \ref{thm:TgammaBallBeh} \ref{items:ballLLN}, it is enough to show the
integrability of $\tau_1$ with respect to $P[\, \cdot \, \vert \, D=\infty].$
On the other hand, in order to deduce Theorem \ref{thm:TgammaBallBeh} \ref{items:averagedDonsker},
an essential part of the proof is to establish the square integrability of $\tau_1$ with respect
to $P[\, \cdot \, \vert \, D = \infty]$ (see also Theorem 4.1 in \cite{Sz-00}).
Both of these integrability conditions are a direct consequence of the following recent result of Berger.

\begin{theorem} (\cite[Prop. 2.2]{Be-12}) \label{eqs:renewalTimesTailUB} 
 Let \ref{items:UEassumption}, \ref{items:IIDassumption}, and 
 $\Pbox_M \vert l$ be fulfilled for some $l \in \mathbb S^{d-1}$ and $M > 15d+5.$
 Then, for $d \ge 4$
 and every $\alpha < d$ one has that
 $$
 P_0[\tau_1 \ge u] \le \exp \{-(\log u)^\alpha\}
 $$
 for all $u$ large enough.
\end{theorem}

In the plain nestling case, this asymptotics is very close to being optimal
as can be seen by the use of so-called na\"ive traps (see proof of \cite[Thm. 2.7]{Sz-00} for a more restricted version of these traps and \cite{Sz-04} also).
These correspond to balls within which the local drift points in the direction of the origin,
see Figure \ref{fig:naiveTraps} as well. Using such traps one gets the following.

\begin{figure}[h]
\begin{center}
\Huge


\psfrag{0}{$0$}

\scalebox{.5}{\includegraphics{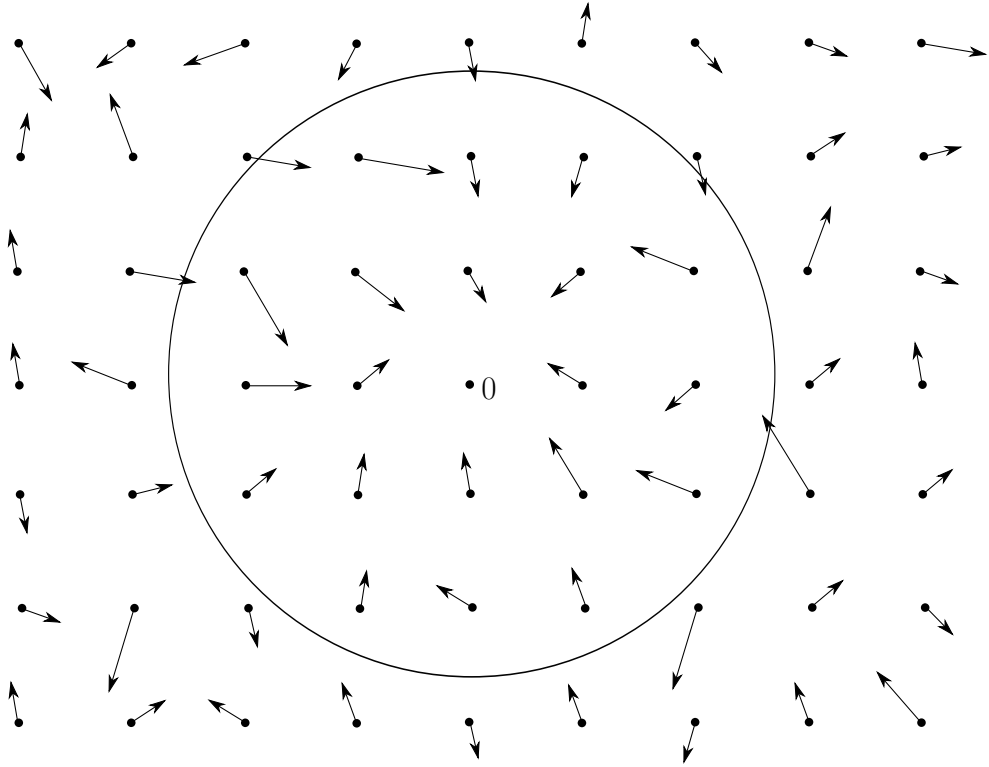}}
\caption{\sl 
A realization of a naive trap: The local drifts within the ball point towards the origin whereas the local drifts
outside the ball are arbitrary.}
\label{fig:naiveTraps}
\end{center}
\end{figure}

\begin{theorem} (\cite[Thm. 2.7]{Sz-00}, \cite{Sz-04}) 
 Assume \ref{items:UEassumption}, \ref{items:IIDassumption}, and
 $\Pbox_M \vert l$ to be fulfilled for some $l \in \mathbb S^{d-1}$ and $M > 15d+5.$ Then, for $d \ge 2$ there exists a constant $C$ such that
 one has 
 $$
 P_0[\tau_1 \ge u] \ge \exp \{-C (\log u)^d\}
 $$
 for all $u$ large enough.

\end{theorem}
As a corollary to Theorem \ref{eqs:renewalTimesTailUB}, 
Berger obtained the following large deviations upper bound, essentially matching Sznitman's
lower bound for the nestling case in \cite[Section 5]{Sz-00}.
In this result, we write $v$ for the $P_0$-a.s. non-zero limit of
$X_n/n,$ cf. Theorem \ref{thm:LLN} \ref{items:ballLLN}.

\begin{theorem} (\cite{Be-12})
 Let \ref{items:UEassumption} and \ref{items:IIDassumption} be fulfilled. Assume furthermore that
 $d \ge 4$ and that 
$\Pbox_M \vert l$ is fulfilled for some $l \in \mathbb S^{d-1}$ and $M > 15d+5.$ Then for $\alpha \in (0,d),$
 $y \in \{tv \, : \, t \in [0,1)\},$ and $\varepsilon \in (0, \vert y-v \vert_2),$ one has
 $$
 P_0 [\vert X_n/n - y \vert_2 < \varepsilon ] < \exp \{-(\log n)^\alpha \}
 $$
 for all $n$ large enough.
\end{theorem}

\begin{remark}
 In the results of \cite{Be-12}, one has the standing assumption that $d \ge 4.$ This assumption
 (in combination with $(T')$) is used
 in order to deduce that, on $\mathbb P$-average,
two independent random walks in the same environment do not meet too often.
 It is plausible that a refinement of the methods in \cite{Be-12} might still yield corresponding results
 in $d =3;$ however, it seems that for the case $d=2$ one essentially needs some further new ideas.
\end{remark}

\section{How to check $(T')$ on finite boxes}
The conditions $(T)_\gamma$ in any of the formulations of Theorem \ref{thm:alternativeT}, 
as well as the condition $\Pasymp,$ are asymptotic in nature
and therefore generally not easy to check. 
In this context, the effective criterion introduced by Sznitman \cite{Sz-02} proves to be a helpful tool for
checking these conditions on finite boxes already.
It can be seen as an analog to the ballisticity conditions of Solomon (cf. Theorem
\ref{thm:oneDimCharacBallist})
in higher dimensions.\footnote{Note that, while
the condition $\Pbox_M$ of Definition \ref{defs:Pbox} also is effective in the sense that it can be checked on finite boxes,
the proof that it implies $(T')$ takes advantage of the effective criterion (cf. Definition \ref{defs:Pbox}
and Theorem \ref{thm:polynomial}) --- we therefore do introduce this  criterion here.}

In order to introduce this criterion, for positive numbers $L,$ $L'$ and $\widetilde{L}$ as well as a space rotation $R$
around the origin we define the 
\begin{align*}
\text{{\it box specification} }{\mathcal{B}}(R, L, L', \widetilde{L})
\text{ as the box }
B:= \big \{x\in\mathbb Z^d:x\in R((-L,L') \times (-\widetilde{L}, \widetilde{L})^{d-1}) \big\}.
\end{align*}
 Recalling the notation of \eqref{eqs:extBd}, we introduce
$$
\rho_{\mathcal{B}}(\omega) := \frac{P_{0,\omega} [{H_{\partial B}} \not= H_{\partial_+ B}]}{P_{0,\omega} [H_{\partial B} =H_{\partial_+ B}]},
$$
where for a subset $A \subset \mathbb Z^d,$ we use the notation
$$
H_A := \inf \{n \ge 0 \, : \, X_n \in A\},
$$
as well as
$$
\partial_+ B := \big \{x \in \partial B : R(e_1) \cdot x \geq L', \vert R(e_j) \cdot x \vert_2 < \widetilde{L} \; \forall j \in \{2, \dots, d\} \big \}.
$$
We will sometimes write $\rho$ instead of $\rho_{\mathcal{B}}$ if the box we refer to is clear from the context. 
\begin{definition}
 Given $l\in\mathbb{S}^{d-1}$, the {\it effective
  criterion with respect to $l$} is satisfied if
for some $L > c_1$ and $\widetilde{L} \in [3\sqrt{d}, L^3),$ we have that
\begin{equation} \label{effectiveCritInf}
\inf_{{\mathcal{B}}, a} \Big\{ c_2 \Big(\ln \frac{1}{\kappa} \Big)^{3(d-1)} \widetilde{L}^{d-1} L^{3(d-1)+1} \mathbb E [ \rho_{\mathcal{B}}^a ] \Big\} < 1.
\end{equation}
Here, when taking the infimum,
$a$ runs over $[0,1]$ while ${\mathcal{B}}$ runs over the 
\begin{equation} \label{eqs:spex}
\text{box specifications }
{\mathcal{B}}(R, L-2, L+2, \widetilde{L}) \text{ with $R$ a rotation around the origin such that $R(e_1) = l.$}
\end{equation}
Furthermore, $c_1$ and $c_2$ are dimension dependent constants.
\end{definition}

\bigskip
The effective criterion is of significant importance 
due to the combination of the facts that
it can be checked on finite boxes (in comparison to $(T')$ which is asymptotic in nature)
and that it is equivalent to $(T'),$ cf. Theorem \ref{thm:SzEffCrit} below.

\begin{theorem}[\cite{Sz-02}] \label{thm:SzEffCrit}
Let \ref{items:IIDassumption} and \ref{items:UEassumption} be fulfilled.
 Then for each $l \in \mathbb{S}^{d-1}$  the following
conditions are equivalent.
\begin{enumerate}
\item The effective criterion with respect  to $l$ is satisfied.

\item
$(T')|l$   is satisfied.

\end{enumerate}
\end{theorem}
In the proof of Theorem \ref{thm:SzEffCrit},
the estimate \eqref{effectiveCritInf} serves as a seed estimate for an involved 
multi-scale renormalization scheme. We refer to the original source for the lengthy proof of this fundamental result,
and to p. 239 ff. of \cite{Sz-04} for a reasonably detailed proof sketch.

\section{Interrelation of stretched exponential ballisticity conditions} \label{sec:ballConds}

While a priori $(T)_\gamma$ is a weaker condition the smaller $\gamma$ is,
 Sznitman \cite{Sz-02} showed
that for each $\gamma\in (0.5,1),$ the conditions
$(T)_\gamma$ and $(T')$ are equivalent. This equivalence has been extended by Drewitz and Ram\'irez \cite{DrRa-09b} to some dimension dependant interval 
$(\gamma_d,1),$ with $\gamma_d \in (0.366, 0.388),$ for all $d \ge 2.$
Furthermore, it has been conjectured (see p. 227 in \cite{Sz-04}) that
\begin{align}
\begin{split} \label{eqs:TgammaEquiv} 
\text{the conditions } (T)_\gamma \vert l \text{ are equivalent for all }
\gamma \in (0,1].
\end{split}
\end{align}

\begin{theorem} [\cite{DR-10,BDR-12}] \label{thm:TgammaEquiv} 
Assume $d \ge 2,$
\ref{items:UEassumption} and \ref{items:IIDassumption} to hold. Then, for $l \in \mathbb S^{d-1},$ the conditions 
$(T)_\gamma \vert l,$ $\gamma \in (0,1),$ are all equivalent.
\end{theorem}

\begin{question}
It is still not known if $(T')$ is actually equivalent to condition $(T);$ however, in some sense there is
 not missing ``too much'' in some sense (see \cite[Prop. 2.3]{Sz-02}).
\end{question}

According to Theorem \ref{thm:TgammaEquiv}, in order to check $(T'),$ it is sufficient to check $(T)_\gamma$
for any $\gamma$ small enough but positive.
As alluded to before already, we will see in the next section that 
it is sufficient to establish the polynomial conditions $\Pasymp_M$ or $\Pbox_M$ for $M$ large enough.

\section{The condition $\Pbox_M$}
%

The main result of this section
will be that of \cite{BDR-12}, namely that for $M$ large enough, $\Pbox_M$  already implies the conditions $(T)_\gamma$
and hence all its consequences such as  ballistic behavior and an invariance principle.

We will be guided by the presentation in \cite{BDR-12} --- however, we will omit a significant share of
the more technical parts of the proof and try to give a less rigorous and more intuitive description instead.
  
The main result of this section is the following.

\begin{theorem}[\cite{BDR-12}] \label{thm:polynomial}
Assume $d \ge 2,$ \ref{items:IIDassumption} and \ref{items:UEassumption} to be fulfilled.
Let $l \in \mathbb{S}^{d-1}$ and assume that $(\mathcal{P}^*)_M \vert l$ or  $(\mathcal{P})_M \vert l$ holds 
for some $M> 15d+5.$
Then $(T') \vert l$ holds.
\end{theorem}

\begin{remark}
The condition $M > 15d +5$ looks quite arbitrary, and is indeed not the weakest condition possible. However, since with the
methods we used it does not seem possible to significantly weaken this condition, we refrain from trying to do so.
\end{remark}

We are going to introduce
some of the  notation needed for the proof of Theorem \ref{thm:polynomial} as well as give two
propositions that play a fundamental role in the proof. 

Let 
\begin{equation} \label{eqs:cZeroDef}
 c_3 = \exp \Big\{100 + 4 d (\ln \kappa)^2 \Big\},
\end{equation}
let
$N_0 \ge c_3$ be an even integer, and
set $N_{-1}:=2N_0/3.$
Using the notation
\begin{equation} \label{eqs:pi}
\pi_l: \mathbb R^d \ni x \mapsto (x \cdot l) \,l  \in \mathbb R^d
\end{equation}
to denote the orthogonal projection on the space $\{\lambda l : \lambda \in \mathbb R\},$
we introduce
the box
\begin{equation}  \label{eqs:boxes}
B := \Big \{ y \in \mathbb Z^d : -\frac{N_0}{2} < (y-x) \cdot l < N_0, \vert \pi_{l^\bot}(y-x) \vert_\infty
< 25N_0^{3} \Big \},
\end{equation}
as well as their frontal parts
\begin{equation} \label{eqs:Btilde}
 \widetilde B := \left\{ y \in \mathbb Z^d : N_0- N_{-1} \leq (y-x) \cdot
l < N_0, \vert \pi_{l^\bot}(y-x) \vert_\infty <
N_0^{3} \right\}.
\end{equation} 
In addition, we define
\begin{equation} \label{eqs:partialPlusDef}
\partial_+ B := \{y \in \partial B : (y-x) \cdot l \geq N_0\}.
\end{equation}

To simplify notation, throughout we will denote a typical box of scale $k$ by $B_k$, and
its middle frontal part by $\widetilde B_k$.

\begin{figure}[h]
\begin{center}
\LARGE

\psfrag{tildeBk}{$\widetilde{B}_0$}
\psfrag{Bk}{$B_0$}
\psfrag{hatv}{$l$}
\psfrag{partialPlus}{$\partial_+ B_0$}
\psfrag{blN}{$bN_0$}
\psfrag{lN}{$N_0$}

\scalebox{.75}{\includegraphics{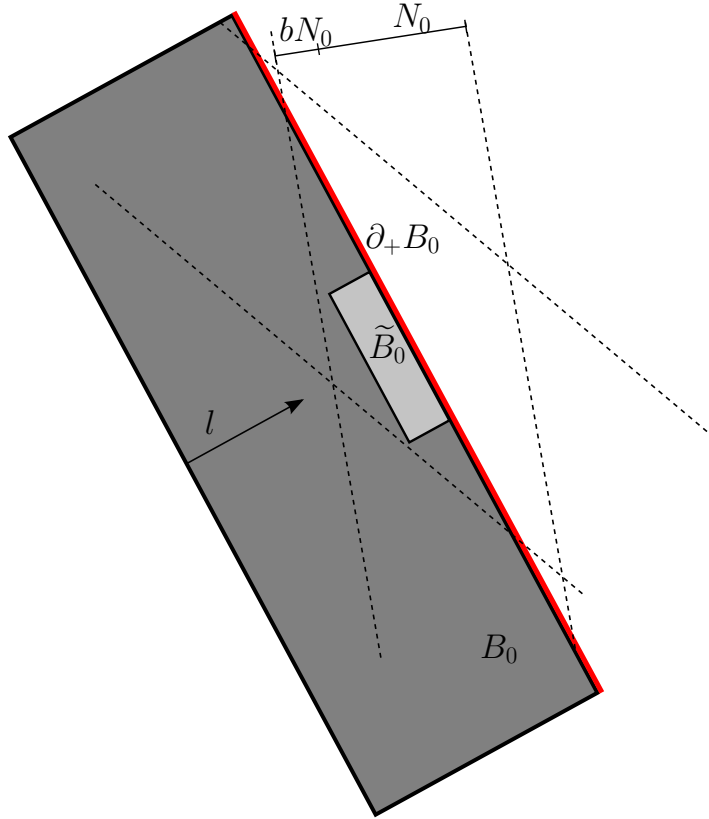}}
\caption{\sl 
A box $B_0$ and the middle frontal part $\widetilde B_0$;
the dashed lines illustrate the slabs from the definition of $\Pasymp_M \vert l,$
shifted by some $x \in \widetilde{B}_0;$
 it is visually apparent here how condition $\Pasymp$ implies condition $\Pbox.$
}
\label{fig:boxAndFront}
\end{center}
\end{figure}  

%
%
%
%
\begin{definition} \label{defs:Pbox}
Let $l \in \mathbb S^{d-1}$ and $M > 0.$
We say that $\Pbox_M \vert l$ is fulfilled if
\begin{align} \label{eqs:singleBound}
 \sup_{x \in \widetilde{B}_0} P_x \big[H_{\partial B_0} \ne H_{\partial_+ B_0} \big] < N_0^{- M}
\end{align}
holds for some $N_0 \ge c_3.$
\end{definition}



\section{An intermediate condition between $\Pbox_M$ and $(T)_\gamma$}
 We need a little further notation for stating this result in particular.
To start with, for a given generic $l=l_1 \in \mathbb S^{d-1},$ we choose $l_2, \ldots, l_d$ arbitrarily in such a way that
$l_1, \ldots, l_d$ forms an orthonormal basis of $\mathbb R^d.$

 For $L > 0,$ define
 $$
 \mathcal{D}_L^l := \Big \{ x \in \mathbb Z^d : -L \leq x  \cdot l \leq 10L, \, \vert x \cdot l_k \vert \le \frac{ L^3 \ln \ln L}{\ln L} \, \forall k \in \{2, \ldots, d\} \Big\} 
 $$
 as well as its {\em frontal boundary part}
 $$
 \partial_+ \mathcal{D}_L^l := \Big \{ x \in \partial \mathcal D_L^l : \pi_l(x) \cdot l > 10L, \, \vert x \cdot l_k \vert \le \frac{ L^3 \ln \ln L}{\ln L} \, \forall k \in \{2, \ldots, d\} \Big\}.
 $$
 In the following we will refer to the condition that 
\begin{align}
\begin{split}\label{eqs:superPolDecay}
&\text{for }l' \in \mathbb S^{d-1} \text{ one has }
P_0 \Big[ H_{\partial \mathcal{D}_L^l} <  H_{\partial_+ \mathcal{D}_L^l} \Big]
\leq \exp \big \{ -L^{\frac{(1+o(1))\ln 2}{\ln \ln L}} \big\},
\end{split}
\end{align}
as $L \to \infty.$

\begin{definition} \label{defs:TgammaL}
If
 $\eqref{eqs:superPolDecay}$  holds for all $l'$  in a neighborhood of $l \in \mathbb S^{d-1},$ then we say that
condition $(T)_{\gamma_L} \vert l$
is fulfilled.
\end{definition}

Since $\gamma_L$ tends to $0$ as $L$ tends to infinity, one observes that the condition $(T)_{\gamma_L}$ is weaker than 
$(T)_\gamma$ for any $\gamma > 0.$

On the other hand, while the condition $(T)_{\gamma_L}$ is a priori stronger than all of the polynomial conditions 
$\Pasymp_M,$
$M > 0,$
it can be shown that it is a consequence of $\Pasymp_M$ once $M$ is chosen large enough.
 This is the content of Proposition \ref{prop:superPolDecay} below.

\section{Strategy of the proof of Theorem \ref{thm:polynomial}}

Using Theorem \ref{thm:SzEffCrit}, we observe that in order to prove Theorem \ref{thm:polynomial}, it is sufficient to establish the effective criterion departing
from $\Pbox_M \vert l$ with $M$ large enough.
On a heuristic level, we will do so via two renormalization schemes:
\begin{enumerate}
 \item 
The first one starts with assuming condition $\Pbox_M \vert l$ for some $M$ 
large enough and derives the intermediate condition $(T)_{\gamma_L}$ introduced in Definition
\ref{defs:TgammaL}.

\begin{proposition}[Sharpened averaged exit estimates] \label{prop:superPolDecay}
Assume \ref{items:IIDassumption} and \ref{items:UEassumption} to be fulfilled.
Let $M> 15d+5,$ $l\in\mathbb S^{d-1},$ and assume that condition
  $(\mathcal{P})_M|l$ 
is satisfied. Then 
$(T)_{\gamma_L} \vert l$ holds.
\end{proposition}

We will not give the technically involved proof of this result and refer to the original source
\cite{BDR-12} instead.

\item

The second renormalization step supplies us with the following large deviations result.

\begin{proposition}[Weak atypical quenched exit estimates, \cite{BDR-12}] \label{prop:quenchedExitEstimates}
Let $d \ge 2$ and assume \ref{items:IIDassumption} and \ref{items:UEassumption} to be fulfilled and let
$(T)_{\gamma_L} \vert l$
hold.
Then for 
$\epsilon(L):=\frac{1}{(\ln\ln L)^2},
$ and any function $\beta:(0,\infty)\to (0,\infty),$ 
one has that
\begin{equation} \label{eqs:quenchedExit}
\mathbb P \Big[P_{0,\omega}  [H_{\partial B}= H_{\partial_+ B}] \leq \frac{1}{2} \exp \big\{ -c_1 L^{\beta(L)} \big\}
 \Big]
\le 5^d \frac{e}{\lceil L^{\beta(L)-\epsilon(L)}/5^d \rceil!},
\end{equation}
where $B$ is a box specification as in \eqref{eqs:spex} with $\widetilde L = L^3-1,$ and
\begin{equation} \label{eqs:c4Def}
c_4 := - 2d \ln \kappa > 1.
\end{equation}
\end{proposition}


This result is much less technical to prove, but nevertheless we refer to \cite{BDR-12} for its proof in order
not to lose the principal thread of these notes.

\end{enumerate}

We do, however, mention that in dimensions $d \ge 4,$ Proposition \ref{prop:quenchedExitEstimates} can be 
strengthened significantly as follows:
\begin{theorem} [Atypical quenched exit estimates, \cite{DR-10}] \label{thm:slabExitDistConjThm}
 Let $d \geq 4,$  and assume \ref{items:IIDassumption}, \ref{items:UEassumption}, and $(T)_\gamma \vert l$ to hold for some
 $\gamma \in (0,1),$ $l \in \mathbb{S}^{d-1}.$ Fix $c > 0$ and $\beta \in (0,1).$
 Then there exists a constant $C > 0$ such that for all $\alpha \in (0, \beta d),$
 \begin{equation*}
 \limsup_{L \to \infty} L^{-\alpha} \log \mathbb P \Big[
P_{0,\omega}  [H_{\partial B}= H_{\partial_+ B}] \leq e^{-cL^\beta} \Big] < 0,
\end{equation*}
where $B$ is a box specification as in \eqref{eqs:spex} with $\widetilde L = CL.$
\end{theorem}
The proof of this result is significantly more involved than that of Proposition \ref{prop:quenchedExitEstimates}.
Note that this theorem is very close to being
optimal in the sense that its conclusion will not hold in general for
$\alpha >\beta d.$ In fact, for plain nestling RWRE,
this can be shown by the use of na\"ive traps
introduced above.

For the purpose of proving  Theorem \ref{thm:polynomial}, however, 
Proposition \ref{prop:quenchedExitEstimates} is sufficient.

\section{Proof of Theorem \ref{thm:polynomial} assuming Propositions \ref{prop:superPolDecay} and \ref{prop:quenchedExitEstimates}}

In this section we demonstrate how Propositions \ref{prop:superPolDecay}
and \ref{prop:quenchedExitEstimates} can be employed in order to establish the effective criterion. 
We will do so by rewriting
$
\mathbb E [ \rho_{\mathcal{B}}^a]
$
of \eqref{effectiveCritInf}
as a sum of terms typically of the form
\begin{equation} \label{eqs:typicalSummand}
\mathcal{E}_j := 
\mathbb E \Big[ \rho^a_{\mathcal{B}}, \frac12 \exp\big\{ -c_4 L^{\beta_{j+1}} \big\}
 < P_{0,\omega}[{H_{\partial B}} = H_{\partial_+ B}] \leq \frac12 \exp \big \{- c_4 L^{\beta_j} \big \} \Big]
\end{equation}
with $\beta_{j+1} > \beta_j.$

Generally,
the lower bound  on $P_{0,\omega}[{H_{\partial B}} = H_{\partial_+ B}]$ in \eqref{eqs:typicalSummand}
 yields a control on the integrand $\rho^a_{\mathcal B}$ from above,
while the upper bound enforces an atypical behavior which will be exploited using Proposition \ref{prop:quenchedExitEstimates}.
The interplay of the upper bound of the integrand thus obtained with the estimate from Proposition \ref{prop:quenchedExitEstimates}
will then determine the asymptotics we obtain for $\mathcal E_j$
(cf. also Lemma \ref{II} below). 

Our proof of Theorem \ref{thm:polynomial} goes along establishing the effective criterion. We do so by a subtle decomposition of the expectation
occurring in \eqref{effectiveCritInf} into several summands, and in the following
we will give some basic lemmas that will prove useful in estimating each of these summands.

For that purpose, we define the quantities
\begin{equation} \label{eqs:gammaDef}
 \beta_1(L)  := \frac{\gamma_L}{2} = \frac{\ln 2}{2 \ln \ln L},
\end{equation}
\begin{equation} \label{eqs:aDef}
a:=L^{-\gamma_L/3},
\end{equation}
and write $\rho$ for $\rho_{\mathcal B}$ with
some arbitrary box specification of \eqref{eqs:spex} with $\widetilde L = L^3-1.$ We split
$
\mathbb E [ \rho^a ]
$
according to 
\begin{align}
\label{decomp}
 E [\rho^a] = \mathcal{E}_0 + \sum_{j=1}^{n-1} \mathcal{E}_j+\mathcal{E}_n,
\end{align}
where 
\begin{equation*} 
n:=n(L):=\Big \lceil \frac{4(1-\gamma_L/2)}{\gamma_L} \Big \rceil + 1,
\end{equation*}
$$
\mathcal{E}_0 := \mathbb E \Big[ \rho^a, P_{0,\omega}[{H_{\partial B}} =
H_{\partial_+ B}] > \frac12 \exp \big\{-c_4 L^{\beta_1} \big\} \Big],
$$
$$
\mathcal{E}_j := 
\mathbb E \Big[ \rho^a, \frac12 \exp\big\{ -c_4 L^{\beta_{j+1}} \big\}
 < P_{0,\omega}[{H_{\partial B}} = H_{\partial_+ B}] \leq \frac12 \exp \big \{- c_4 L^{\beta_j} \big \} \Big]
$$
for $j \in \{1, \ldots, n-1\},$
and
$$
\mathcal{E}_n := \mathbb E \Big[ \rho^a, P_{0,\omega}[{H_{\partial B}} = H_{\partial_+ B}] \le \frac12 
\exp \big \{- c_4 L^{\beta_n} \big\} \Big],
$$
with parameters 
\begin{equation} \label{eqs:betaDef}
\beta_j(L) := \beta_1(L) + (j-1) \frac{\gamma_L}{4},
\end{equation}
for $2\le j\le n(L);$  for the sake of brevity we may sometimes omit the dependence on $L$  of the parameters
if that does not cause any confusion.
Furthermore, in order to verify equality \eqref{decomp},
note that due to the uniform ellipticity assumption \ref{items:UEassumption}
and the choice of $c_4$ (cf. \eqref{eqs:c4Def}),
one has for $\mathbb P$-a.a. $\omega$ that
$$
P_{0,\omega}[{H_{\partial B}} = H_{\partial_+ B}] > e^{-c_4 L},
$$
as well as that
\begin{equation*} 
\beta_{n} >1.
\end{equation*}
To bound $\mathcal{E}_0$ we employ the following lemma.
\begin{lemma}
\label{I} Let $(T)_{\gamma_L}$ be fulfilled. Then
\begin{equation} 
\nonumber
\mathcal{E}_0
\leq \exp \big\{ c_4 L^{\gamma_L/6} - L^{(1+o(1))\gamma_L/2} \big\},
\end{equation}
as $L \to \infty.$
\end{lemma}
\begin{proof}
Jensen's inequality yields
\begin{align*}
\mathcal{E}_0
\leq 2\exp\big\{ c_4 L^{\beta_1-\gamma_L/3} \big\}
P_0[{H_{\partial B}} \not= H_{\partial_+ B}]^a.
\end{align*}
Using \eqref{eqs:gammaDef}
 in combination with  $(T)_{\gamma_L}$ we obtain the desired result.
\end{proof}

To deal with the middle summand in the right-hand side of (\ref{decomp}), we
use the following lemma.

\begin{lemma} \label{II} 
Let \ref{items:IIDassumption} and \ref{items:UEassumption} be fulfilled
and assume $(T)_{\gamma_L} \vert l$ to hold. Then for all $L$ large enough we have uniformly in
${j \in \{1, \ldots, n-1\}}$ that
\begin{equation*} 
\mathcal{E}_j
\leq 2 \cdot  5^d \exp \big\{c_4 L^{\beta_{j+1}-\gamma_L/3}\big\} \frac{e}{ \lceil L^{\beta_j-\epsilon(L)}/5^d \rceil!}.
\end{equation*}
\end{lemma}
\begin{proof}
Using Markov's inequality, for $j \in \{1, \ldots, n-1\}$ we obtain the estimate
\begin{align} \label{(II)Est}
\mathcal{E}_j\leq 2\exp \big\{c_4 L^{\beta_{j+1} - \gamma_L/3}\big\}
\mathbb P \Big[ P_{0,\omega}[{H_{\partial B}} = H_{\partial_+ B}] \leq \frac12  \exp \big\{ -c_4 L^{\beta_{j}} \big\} \Big].
\end{align}
Thus, due to Proposition
\ref{prop:quenchedExitEstimates},
the probability on the right-hand side of (\ref{(II)Est}) can be estimated from
above by
$$
 5^d \frac{e}{\lceil L^{\beta_j-\epsilon(L)}/5^d \rceil!}.
$$
\end{proof}
When it comes  to the term $\mathcal{E}_n$ in (\ref{decomp}) we note that it vanishes
because of the choice of $c_4.$

\begin{proof}[Proof of Theorem \ref{thm:polynomial}]
It follows from  Lemmas \ref{I}, \ref{II}, the choice of parameters in \eqref{eqs:gammaDef} to \eqref{eqs:aDef}
and \eqref{eqs:betaDef}, and the fact that $\mathcal E_n$ vanishes, that 
for $L$ large enough, (\ref{decomp}) can be bounded from above by 
$$
\exp \big\{ c_4 L^{\gamma_L/6} - L^{(1+o(1))\gamma_L/2} \big\} +  
2 \cdot 5^d n(L) 
\max_{1 \le j \le n(L)-1} \Big( \exp \big\{c_4 L^{\beta_{j+1}-\gamma_L/3}\big\} \frac{e}{\lceil L^{\beta_j-\epsilon(L)}/5^d \rceil!} \Big).
$$
Thus, we see that for our choice of parameters, \eqref{decomp} tends to zero faster than any polynomial in $L.$
Hence, due to \eqref{effectiveCritInf}, the effective criterion holds and Theorem \ref{thm:SzEffCrit} then yields the desired result.
\end{proof}

\section{Relation between directional transience and slab exit estimates}

The aim of this subsection is to show how the condition of directional transience relates to
slab exit estimates such as condition $(\mathcal{P}^*)_M.$

\begin{lemma} \label{lem:PImpliesTransience}
Let $l\in\mathbb S^{d-1},$ and suppose that $(\mathcal{P}^*)_M |l$
is satisfied. 

\begin{enumerate}
\item
There exists a constant $C$ such that
\begin{equation} \label{eqs:decayEst}
P_0 \big[ T^{-l}_{-L} \circ \theta_{T_{2L}^l} \leq T_{4 L}^l \circ \theta_{T_{2 L}^l} \big]
\leq C L^{d-1-M}
\end{equation}
for all $L \in \mathbb N.$

\item
If $M > d,$ 
then
$P_0$-a.s. the random walk $X$ is transient in direction $l,$ i.e. $P_0[A_l] = 1.$

\end{enumerate}

\end{lemma}
\begin{proof}
\begin{enumerate}
\item
The general idea of this proof is taken from a stretched exponential analog \cite[Theorem 2.11]{Sz-02}. Note that
$$
\big\{ T_L^l < T_L^{-l} \big\} \subset \Big \{\limsup_{n \to \infty} X_n \cdot l \geq L \Big\}.
$$
Due to condition $\Pasymp_M |l$, the probability with respect to $P_0$
of the left-hand side tends to $1$ as $L \to \infty,$
which implies
\begin{equation} \label{eqs:ASTransience}
P_0 \Big[ \limsup_{n \to \infty} X_n \cdot l = \infty \Big] = 1.
\end{equation}

Now choose $l_1,\ldots, l_d\in\mathbb{S}^{d-1}\cap V_l$ (with $V_l$ denoting the neighborhood associated to $l$
in the definition of $(\mathcal{P}^*)_M |l$, see Definition \ref{defs:Pasymp}) to be linearly
independent. If furthermore $l_1,\ldots, l_d$ are chosen sufficiently close to $l,$
setting $l_0:= l$ there exists $\delta >0$ such that for
\begin{equation} \label{eqs:DeltaDef}
\Delta_{L} := \big \{ x \in \mathbb Z^d : -\delta L \leq x \cdot l_j \leq L \; \forall 0 \leq j \leq d \big \} 
\end{equation}
and
$$
\partial_+ \Delta_{L} := \Big \{ x \in \partial \Delta_L : \max_{0 \leq j \leq d} x \cdot l_j  > L
\text{ and } \min_{0 \leq j \leq d} x \cdot l_j \geq -\delta L  \Big \},
$$
we have
\begin{equation} \label{eqs:deltaBound}
2\delta L \leq \min \{ x \cdot l : x \in \partial_+ \Delta_L\}. 
\end{equation}
Now due to \eqref{eqs:ASTransience}, we infer that $T_L^l$ is finite $P_0$-a.s. and
hence $\theta_{T_L^l}$ is well-defined for all $L > 0.$
Thus, we get using 
the strong Markov property at time $T^l_{2\delta L}$
(applied to the quenched walk) in combination with the translation invariance of $\mathbb P$ and \eqref{eqs:deltaBound}, that
\begin{align}
 &P_0 \big[ T^{-l}_{-\delta L} \circ \theta_{T_{2\delta L}^l} \leq T_{4\delta L}^l \circ \theta_{T_{2\delta L}^l} \big]
\nonumber \\
&\leq P_0[T_{\partial \Delta_L} < T_{2\delta L}^l]
+ P_0 \big[ T^{-l}_{-\delta L} \circ \theta_{T_{2\delta L}^l}
\leq T_{4\delta L}^l \circ \theta_{T_{2\delta L}^l}, T_{\partial \Delta_L} \geq T_{2\delta L}^l \big] \nonumber\\
&\leq \sum_{j=0}^d P_0[T_{\delta L}^{-l_j} < T_{L}^{l_j}]
+ C L^{d-1} P_0 \big[ T_{\delta L}^{-l} \leq T_{2\delta L}^l \big]. \label{eqs:slabExitEst}
\end{align}
To obtain the last line we used the fact that, since $l_1, \ldots, l_d$ form a basis,
$
\vert \{x \in \Delta_L : x\cdot l \in (2\delta L, 2\delta L + 1] \} \vert \leq C L^{d-1}
$
holds.
Since furthermore  $(\mathcal{P}^*)_M |l$ is fulfilled
 we can estimate \eqref{eqs:slabExitEst}
from above by $CL^{d-1-M}$ which proves the first assertion of the lemma.

\item
 Using this result 
in combination with
the assumption that $M > d,$ Borel-Cantelli's lemma
yields that $P_0$-a.s., for eventually all $L \in \mathbb N,$
$$
T_{4L}^{l} \circ \theta_{T_{2L}^l} < T_{-L}^{-l} \circ \theta_{T_{2 L}^l}.
$$
This implies that
$
P_0[\lim_{n \to \infty} X_n \cdot l = \infty] = 1.
$

\end{enumerate}

\end{proof}

We have the following corollary on the relation between transience and the conditions $\Pasymp_M.$
\begin{corollary}
The implications
\begin{align}
\Pasymp_M \vert l \text{ for some }M > d 
\quad \implies \quad 
P_0[A_{l'}]=1 \, \forall \, l' \text{ in a neighborhood } V_l \text{ of } l 
 \quad \implies \quad \Pasymp_0 \vert l
\end{align}
hold true.
\end{corollary}

\begin{proof}
The first implication is a direct consequence of Lemma \ref{lem:PImpliesTransience}. To obtain the second implication note
that if $P_0[A_{l'}] = 1$ for all $l' \in V_l,$ then we have
$$
P_0 \big[ H^{-l'}_{bL} < H^{l'}_{L} \big] \le 
P_0 \big[A_{l'}, H^{-l'}_{bL} < \infty] \to 0, \quad \text{ as } L \to \infty,
$$
where we used that $P_0[\cdot \, \vert \, A_{l'}]$-a.s. one has $\inf_{n \in \mathbb N} X_n \cdot l' \in (-\infty, 0].$
\end{proof}

\begin{remark}
The above corollary immediately leads to two questions:
\begin{enumerate}
\item
Which is the minimal $M$ for which the first implication holds?

\item
Can $\Pasymp_0$ on the right-hand side of the implications be replaced by $\Pasymp_M$ for some $M>0,$ and if
so, what is the maximal $M$?
\end{enumerate}
These questions are intimately connected to 
open question \ref{qn:transBall}.
\end{remark}

\section{Ellipticity conditions for ballistic behavior} \label{sec:ellCond}
We have seen in Chapter \ref{chap:I} that there can exist
elliptic random walks which are transient in a given direction
but which are not ballistic. On the other hand, Proposition 2.17
of this chapter shows that at least some condition on the
moments of the jump probabilities of the random environment
should be asked if we  expect to extend the results of
this chapter.

\begin{definition} Consider a RW in an environment $\mathbb P$.
We say that $\mathbb P$ satisfies the ellipticity condition
$(E)_\beta$ if there exist positive parameters $\{\beta_e:e\in U\}$
such that

$$
2\sum_{e\in U}\beta_e-\sup_{e'}(\beta_{e'}+\beta_{-e'})>\beta
$$
and

$$
\mathbb E\left[e^{\sum_e\beta_e\log\frac{1}{\omega(0,e)}}\right]<\infty.
$$
If in addition there exists a $\bar\beta$  such that
$\beta_e=\bar\beta$ for $e$ such that $e\cdot\hat v\ge 0$ (recall that $\hat v$ was the asymptotic direction) while
$\beta_e\le\bar\beta$ for $e$ such that $e\cdot\hat v< 0$, we say that
condition $(E)_\beta$ is satisfied directionally.
Furthermore, whenever there exists an $\alpha>0$ such that

$$
\sup_e\mathbb E\left[\frac{1}{\omega(0,e)^\alpha}\right]<\infty
$$
we say that the law $\mathbb P$ of the environment satisfies
condition $(E')_\alpha$.
\end{definition}

We have the following extension of Theorem \ref{thm:polynomial}
proved in \cite{CR-13}.

\begin{theorem}
\label{cr13} {\bf (Campos-Ram\'\i rez)} Consider a random walk in an i.i.d. environment
which satisfies condition $(E')_\alpha$ for some $\alpha>0$.
Then, if $\Pasymp_M|l$ is satisfied for some $M\ge 15d+5$,
$(T')|l$ is satisfied.
\end{theorem}

Furthermore, we have then the following consequence of Theorem
\ref{cr13} proved in \cite{CR-13}.

\begin{theorem}
{\bf (Campos-Ram\'\i rez)} Consider a random walk in an i.i.d. environment
which satisfies condition $(E)_1$ directionally.
Then, if $\Pasymp_M|l$ 
 is satisfied for $M\ge 15d+5$, the walk is ballistic.
\end{theorem}

\bigskip

\noindent {\bf Acknowledgement:} The final version has benefitted from careful
refereeing. We would also like to thank Gregorio Moreno for
useful comments on a first draft of this text.

  \printindex


\end{document}